\documentclass{amsart}
\usepackage{epsf}
\usepackage[all]{xy}
\usepackage{enumerate}
\usepackage{amssymb}
\usepackage{color}

\newcommand{\ie}{{\em i.e.\ }}
\newcommand{\eg}{{\em e.g.\ }}
\newcommand{\cf}{{\em cf.\ }}

\newtheorem*{theorem}{Theorem}
\newtheorem*{lemma}{Lemma}
\newtheorem*{proposition}{Proposition}
\newtheorem*{corollary}{Corollary}

\theoremstyle{definition}
\newtheorem*{definition}{Definition}

\newtheorem*{example}{Example}

\theoremstyle{remark}
\newtheorem*{remark}{\textbf{Remark}}

\numberwithin{equation}{section}

\newcommand{\internalcomment}[1]{}

\newcommand{\N}{\mathbf{N}}
\newcommand{\Z}{\mathbf{Z}}

\newcommand{\ko}{\: , \;}

\newcommand{\ul}{\underline}
\newcommand{\we}{\wedge}
\newcommand{\che}{\vee}

\newcommand{\ra}{\rightarrow}

\newcommand{\arr}[1]{\stackrel{#1}{\rightarrow}}

\newcommand{\opname}[1]{\operatorname{\mathsf{#1}}}

\newcommand{\Mod}{\opname{Mod}\nolimits}
\newcommand{\Sum}{{\mbox{Sum}}}

\newcommand{\Tria}{{\mbox{Tria}}}

\newcommand{\id}{\mathbf{1}}

\renewcommand{\L}{\mathbf{L}}

\newcommand{\cok}{\opname{cok}\nolimits}
\newcommand{\im}{\opname{im}\nolimits}
\renewcommand{\ker}{\opname{ker}\nolimits}
\newcommand{\colim}{\colim}
\newcommand{\lid}{\varinjlim}

\newcommand{\Mcolim}{\opname{Mcolim}}
\newcommand{\can}{\opname{can}}

\newcommand{\op}[1]{\opname{#1}\nolimits}
\newcommand{\Zy}[1]{\op{Z}^{#1}}
\newcommand{\Bo}[1]{\op{B}^{#1} \,}
\renewcommand{\H}[1]{{H}^{#1}}

\newcommand{\ca}{{\mathcal A}}
\newcommand{\cb}{{\mathcal B}}
\newcommand{\cc}{{\mathcal C}}
\newcommand{\cd}{{\mathcal D}}
\newcommand{\ce}{{\mathcal E}}
\newcommand{\cF}{{\mathcal F}}
\newcommand{\cg}{{\mathcal G}}
\newcommand{\ch}{{\mathcal H}}
\newcommand{\ci}{{\mathcal I}}

\newcommand{\cm}{{\mathcal M}}

\newcommand{\cp}{{\mathcal P}}
\newcommand{\cR}{{\mathcal R}}
\newcommand{\cq}{{\mathcal Q}}
\newcommand{\cs}{{\mathcal S}}

\newcommand{\cu}{{\mathcal U}}
\newcommand{\cv}{{\mathcal V}}

\newcommand{\cx}{{\mathcal X}}
\newcommand{\cy}{{\mathcal Y}}
\newcommand{\cz}{{\mathcal Z}}
\renewcommand{\phi}{\varphi}

\newcommand{\Hom}{\opname{Hom}}

\newcommand{\Ext}{\opname{Ext}}

\newcommand{\cone}{\opname{Cone}\nolimits}

\newcommand{\Fun}{\opname{Fun}}

\begin{document}
\title{Parametrizing recollement data}

\author{Pedro Nicol\'{a}s}
\address{Departamento de Matem\'{a}ticas, Universidad de Murcia, Aptdo. 4021, 30100, Espinardo, Murcia, Espa\~na}

\email{pedronz@um.es}
\thanks{The authors have been partially supported by research projects from the D.G.I. of the Spanish Ministry of Education and the
Fundaci\'{o}n S\'{e}neca of Murcia, with a part of FEDER funds. The first author has been also supported by the MECD grant AP2003-2896.}

\author{Manuel Saor\'{i}n}
\address{Departamento de Matem\'{a}ticas, Universidad de Murcia, Aptdo. 4021, 30100, Espinardo, Murcia, Espa\~na}
\email{msaorinc@um.es}

\date{December 18, 2007}



\begin{abstract}
We give a general parametrization of all the recollement data for a triangulated category with a set of generators. From this we deduce a characterization of when a perfectly generated (or \emph{aisled}) triangulated category is a recollement of triangulated categories generated by a single compact object. Also, we use \emph{homological epimorphisms} to give a complete and explicit description of all the recollement data for (or smashing subcategories of) the derived category of a $k$-flat dg category. In the final part we give a bijection between smashing subcategories of compactly generated triangulated categories and certain ideals of the subcategory of compact objects, in the spirit of H.~Krause's work \cite{Krause2005}. This bijection implies the following weak version of the \emph{generalized smashing conjecture:} in a compactly generated triangulated category every smashing subcategory is generated by a set of Milnor colimits of compact objects.
\end{abstract}

\maketitle

\section{Introduction}

\subsection{Motivations}

The definition of \emph{t-structure} and \emph{recollement} for triangulated categories was given by A. A. Beilinson, J. Bernstein and P. Deligne in their work \cite{BeilinsonBernsteinDeligne} on perverse sheaves. The notion of t-structure is the analogue of the notion of \emph{torsion pair} \cite{Dickson1966, Stenstrom, BeligiannisReiten} for abelian categories. Accordingly, the `triangulated' analogue of a \emph{torsion torsionfree(=TTF) triple} \cite{Jans, Stenstrom, BeligiannisReiten},  still called \emph{TTF triple}, consists of a triple $(\cx,\cy,\cz)$ such that both $(\cx,\cy)$ and $(\cy,\cz)$ are t-structures. In general, TTF triples on a triangulated category $\cd$ are in bijection with (equivalence classes of) ways of expressing $\cd$ as a recollement of triangulated categories, and with smashing subcategories of $\cd$ for instance when it is perfectly generated.

One of the aims of this paper is to give a parametrization of all the TTF triples on triangulated categories of a certain type including those which are well (or even perfectly) generated. This parametrization might result naive but nevertheless, together with B.~Keller's Morita theory for derived categories \cite{Keller1994a}, it yields a generalization of some results of \cite{DwyerGreenlees, Jorgensen2006} and offers an unbounded and abstract version of S. K\"{o}nig's theorem \cite{Konig1991} on recollements of right bounded derived categories of algebras. In a forthcoming paper we will study the problem of descending the parametrization from unbounded to right bounded derived categories, and the corresponding lifting.

The following facts suggest that, in the case of derived categories, a more sophisticated parametrization is possible:
\begin{enumerate}[1)]
\item TTF triples on categories of modules are well understood and a tangible param\-etrization of them was given by J. P. Jans \cite{Jans}.
\item A natural proof of J. P. Jans' theorem uses P. Gabriel's characterization of categories of modules among abelian categories \cite{Gabriel1962}, which is at the basis of Morita theory.
\item P.~Gabriel's characterization admits a `triangulated' analogue which was proved by B. ~Keller, who developed a Morita theory for derived categories of dg categories in \cite{Keller1994a} (and later in \cite{Keller1996, Keller1998b},\dots).
\item It seems that derived categories of dg categories play the r\^ole, in the theory of triangulated categories, that module categories play in the theory of abelian categories.
\end{enumerate}
Then, another aim of this paper is to give a touchable parametrization of TTF triples on derived categories of dg categories by using B.~Keller's theory, and to elucidate their links with H.~Krause's parametrization \cite{Krause2000, Krause2005} of smashing subcategories of compactly generated triangulated categories. For this, we use a generalization of the notion of \emph{homological epimorphism} due to W. Geigle and H. Lenzing \cite{GeigleLenzing}. Homological epimorphisms appear as (stably flat) universal localizations in the work of P. M. Cohn \cite{Cohn}, A. H. Schofield \cite{Schofield}, A. Neeman and A. Ranicki \cite{NeemanRanickiI}, \dots Recently, H.~Krause has studied \cite{Krause2005} the link between homological epimorphisms of algebras and: the \emph{chain map lifting problem}, the \emph{generalized smashing conjecture} and the existence of long exact sequences in algebraic K-theory. Homological epimorphisms also appear in the work of L. Angeleri H\"{u}gel and J. S\'{a}nchez \cite{AngeleriSanchez2007} on the construction of tilting modules induced by ring epimorphisms.

\subsection{Contents}

In section \ref{Notation and preliminary results}, we fix some terminology and recall some results on triangulated categories, emphasizing B.~Keller's work \cite{Keller1994a, Keller1998b} on derived categories of dg categories. We prove, however, some apparently new results which measure the distance between ``compact'' and: `self-compact', ``perfect'', ``superperfect''. In section \ref{Parametrization}, we introduce the notion of \emph{recollement-defining class} (subsection \ref{Recollement-defining classes}) and prove how to find recollement-defining sets in \emph{aisled triangulated categories} (subsection \ref{Recollement-defining sets in aisled categories}) and in perfectly generated triangulated categories (subsection \ref{Recollement-defining sets in perfectly generated triangulated categories}). In subsection \ref{Parametrization of TTF triples on triangulated categories}, recollement-defining sets enable us to parametrize all the TTF triples on a triangulated category with a set of generators, and all the ways of expressing a `good' triangulated category as a recollement of compactly generated triangulated categories. In section \ref{Homological epimorphisms of dg categories}, we introduce the notion of \emph{homological epimorphisms of dg categories}, generalizing the homological epimorphisms of algebras of W. Geigle and H. Lenzing \cite{GeigleLenzing}. We easily prove that this kind of morphisms always induce a TTF triple, which allows us to give several examples of recollements for unbounded derived categories of algebras which were already known for right bounded derived categories (\cf S. K\"{o}nig's paper \cite{Konig1991}). Conversely, we prove that every TTF triple on the derived category of a $k$-flat dg category $\ca$ is induced by a homological epimorphism starting in $\ca$. This correspondence between TTF triples and homological epimorphisms keeps a lot of similitudes with the one accomplished by J. P. Jans \cite{Jans} for module categories. In section \ref{Parametrization for derived categories}, we state a parametrization of smashing sub\-ca\-te\-go\-ries of a compactly generated algebraic triangulated category which uses the main results of subsection \ref{Parametrization of TTF triples on triangulated categories} and section \ref{Homological epimorphisms of dg categories}. Finally, in section \ref{Idempotent two-sided ideals}, we analyse how idempotent two-sided ideals of the subcategory $\cd^c$ of compact objects appear in the description of TTF triples on a compactly generated triangulated category $\cd$. More concretely, in Theorem \ref{From ideals to smashing subcategories} we prove that an idempotent two-sided ideal of $\cd^c$, which is moreover stable under shifts in both directions, always induces a \emph{nicely} described TTF triple on $\cd$. This, together with assertion 2') of H.~Krause's \cite[Theorem 4.2]{Krause2000}, gives a short proof of a result (\cf Theorem \ref{General case}) in the spirit of H.~Krause's bijection \cite[Theorem 11.1, Theorem 12.1, Corollary 12.5 and Corollary 12.6]{Krause2005} between smashing subcategories and special idempotent two-sided ideals. As a consequence (\cf Corollary \ref{General case}), we get the following weak version of the \emph{generalized smashing conjecture}: every smashing subcategory of a compactly generated triangulated category is generated by a set of Milnor colimits of compact objects. Another consequence (\cf Corollary \ref{Algebraic case}) is that, when $\cd$ is algebraic, we recover precisely H.~Krause's bijection.

\subsection{Acknowledgements}

We are grateful to Bernhard Keller for several discussions concerning dg categories.

\section{Notation and preliminary results}\label{Notation and preliminary results}

Unless otherwise stated, $k$ will be a commutative (associative, unital) ring and every additive category will be assumed to be $k$-linear. We denote by $\Mod k$ the category of $k$-modules. Given a class $\cq$ of objects of an additive category $\cd$, we denote by $\cq^{\bot_{\cd}}$ (or $\cq^{\bot}$ if the category $\cd$ is clearly assumed) the full subcategory of $\cd$ formed by the objects $M$ which are \emph{right orthogonal} to every object of $\cq$, \ie such that $\cd(Q,M)=0$ for all $Q$ in $\cq$. Dually for $\ ^{\bot_{\cd}}\cq$. When $\cd$ is a triangulated category, the \emph{shift functor} will be denoted by $?[1]$. When we speak of ``all the shifts'' or ``closed under shifts'' and so on, we will mean ``shifts in both directions'', that is to say, we will refer to the $n$th power $?[n]$ of $?[1]$ for all the integers $n\in\Z$. In case we want to consider another situation (\eg non-negative shifts $?[n]\ko n\geq 0$) this will be said explicitly. We will use without explicit mention the bijection between t-structures on a triangulated category $\cd$ and \emph{aisles} in $\cd$, proved by B.~Keller and D.~Vossieck in \cite{KellerVossieck88}. If $(\cu,\cv[1])$ is a t-structure on a triangulated category $\cd$, we denote by $u:\cu\hookrightarrow\cd$ and $v:\cv\hookrightarrow\cd$ the inclusion functors, by $\tau_{\cu}$ a right adjoint to $u$ and by $\tau^{\cv}$ a left adjoint to $v$.

\subsection{TTF triples and recollements}\label{TTF triples and recollements}

A \emph{TTF triple} on $\cd$ is a triple $(\cx,\cy,\cz)$ of full subcategories of $\cd$ such that $(\cx,\cy)$ and $(\cy,\cz)$ are t-structures on $\cd$. Notice that, in particular, $\cx\ko \cy$ and $\cz$ are full triangulated subcategories of $\cd$. It is well known that TTF triples are in bijection with (equivalence classes of) \emph{recollements} (\cf \cite[1.4.4]{BeilinsonBernsteinDeligne}, \cite[subsection 9.2]{Neeman2001}, \cite[subsection 4.2]{NicolasTesis}). For the convenience of the reader we recall here how this bijection works. If
\[\xymatrix{ \cd_{F}\ar[r]^{i_{*}} & \cd\ar@/_-1pc/[l]^{i^!}\ar@/_1pc/[l]_{i^*}\ar[r]^{j^*} & \cd_{U}\ar@/_-1pc/[l]^{j_{!}}\ar@/_1pc/[l]_{j_{*}}
}
\]
expresses $\cd$ as a recollement of $\cd_{F}$ and $\cd_{U}$, then
\[(j_{!}(\cd_{U}),i_{*}(\cd_{F}),j_{*}(\cd_{U}))
\]
is a TTF triple on $\cd$, where by $j_{!}(\cd_{U})$ we mean the essential image of $j_{!}$, and analogously with the other functors. Conversely, it is straightforward to check that if $(\cx,\cy,\cz)$ is a TTF triple on $\cd$, then $\cd$ is a recollement of $\cy$ and $\cx$ as follows:
\[\xymatrix{ \cy\ar[r]^{y} & \cd\ar@/_-1pc/[l]^{\tau_{\cy}}\ar@/_1pc/[l]_{\tau^\cy}\ar[r]^{\tau_{\cx}} & \cx,\ar@/_-1pc/[l]^{x}\ar@/_1pc/[l]_{z\tau^{\cz}x}
}
\]
Notice that for a TTF triple $(\cx,\cy,\cz)$ the compositions $\cx\arr{x}\cd\arr{\tau^{\cz}}\cz$ and $\cz\arr{z}\cd\arr{\tau_{\cx}}\cx$ are mutually quasi-inverse triangle equivalences (\cf \cite[Lemma 1.6.7]{NicolasTesis}).

\subsection{Generators and infinite d\'{e}vissage}\label{Generators and infinite devissage}

Let $\cd$ be a triangulated category. We say that it is \emph{generated} by a class $\cq$ of objects if an object $M$ of $\cd$ is zero whenever
\[\cd(Q[n],M)=0
\]
for every object $Q$ of $\cq$ and every integer $n\in\Z$. In this case, we say that $\cq$ is a \emph{class of generators} of $\cd$ and that $\cq$ \emph{generates} $\cd$.

If $\cd$ has small coproducts, given a class $\cq$ of objects of $\cd$ we denote by $\Tria_{\cd}(\cq)$ (or $\Tria(\cq)$ if the category $\cd$ is clear) the smallest full triangulated subcategory of $\cd$ containing $\cq$ and closed under small coproducts. We say that $\cd$ satisfies the \emph{principle of infinite d\'{e}vissage} with respect to a class of objects $\cq$ if it has small coproducts and $\cd=\Tria(\cq)$. In this case, it is clear that $\cd$ is generated by $\cq$.

Conversely, the first part of the following lemma states that under certain hypothesis `generators' implies `d\'{e}vissage'.

\begin{lemma}
\begin{enumerate}[1)]
\item Let $\cd$ be a triangulated category with small coproducts and let $\cd'$ be a full triangulated subcategory generated by a class of objects $\cq$.
If $\Tria(\cq)$ is an aisle in $\cd$ contained in $\cd'$, then $\cd'=\Tria(\cq)$.
\item Let $\cd$ be a triangulated category and let $(\cx,\cy)$ be a t-structure on $\cd$ with triangulated aisle.
\begin{enumerate}[2.1)]
\item If $\cq$ is a class of generators of $\cd$, then $\tau^{\cy}(\cq)$ is a class of generators of $\cy$.
\item A class $\cq$ of objects of $\cx$ generates $\cx$ if and only if the objects of $\cy$ are precisely those which are right orthogonal to all the shifts of objects of $\cq$.
\end{enumerate}
\end{enumerate}
\end{lemma}
\begin{proof}
1) Given an object $M$ of $\cd'$ there exists a triangle
\[M'\ra M\ra M''\ra M'[1]
\]
with $M'$ in $\Tria(\cq)$ and $M''$ in $\Tria(\cq)^{\bot_{\cd}}$. Since $M'$ and $M$ are in $\cd'$, then so is $M''$. But $\cd'$ is generated by $\cq$, which implies that $M''=0$ and so $M$ belongs to $\Tria(\cq)$.

2) Left as an exercise.
\end{proof}

\subsection{(Super)perfectness and compactness}\label{(Super)perfectness and compactness}

An object $P$ of $\cd$ is \emph{perfect} (respectively, \emph{superperfect}) if for every countable (respectively, small) family of morphisms $M_{i}\ra N_{i}\ko i\in I$, of $\cd$ such that the natural morphism $\coprod_{I}M_{i}\ra\coprod_{I}N_{i}$ exists the induced map
\[\cd(P,\coprod_{I}M_{i})\ra\cd(P,\coprod_{I}N_{i})
\]
is surjective provided every map
\[\cd(P,M_{i})\ra\cd(P,N_{i})
\]
is surjective. Particular cases of superperfect objects are \emph{compact} objects, \ie objects $P$ such that $\cd(P,?)$ commutes with small coproducts.

The following lemma is very useful. It shows some links between t-structures and (super)perfect and compact objects. First we remind the following definitions. Let $\cd$ be a triangulated category. A contravariant functor $H:\cd\ra\Mod k$ is \emph{cohomological} if for every triangle
\[L\arr{f}M\arr{g}N\ra L[1]
\]
the sequence
\[H(N)\arr{H(g)}H(M)\arr{H(f)}H(L).
\]
is exact. We say that $\cd$ \emph{satisfies the Brown's representability theorem for cohomology} if every cohomological functor $H:\cd\ra\Mod k$ taking small coproducts to small products is representable.

\begin{lemma}
\begin{enumerate}[1)]
\item Let $\cd$ be a triangulated category and let $(\cx,\cy)$ be a t-structure on $\cd$ with $\cx$ triangulated and such that the inclusion functor $y:\cy\hookrightarrow\cd$ preserves small coproducts. If $M$ is a perfect (respectively, superperfect, compact) object of $\cd$, then $\tau^{\cy}M$ is a perfect (respectively, superperfect, compact) object of $\cy$.
\item The adjoint functor argument: Let $\cd$ be a triangulated category with small coproducts and let $\cd'$ be a full triangulated subcategory of $\cd$ closed under small coproducts and satisfying the Brown's representability theorem. In this case $\cd'$ is an aisle in $\cd$. In particular, if $\cp$ is a set of objects of $\cd$ which are perfect in $\Tria(\cp)$, then $\Tria(\cp)$ is an aisle in $\cd$.
\end{enumerate}
\end{lemma}
\begin{proof}
1) is well known \cite[Lemma 2.4]{Neeman1992} for the case of compact objects. Using the adjunction $(y,\tau^{\cy})$ it also follows easily for the case of (super)perfect objects.

2) If
\[\iota: \cd'\hookrightarrow\cd
\]
is the inclusion functor, for an object $M$ of $\cd$ we define the functor
\[H(?):=\cd(\iota(?),M):\cd'\ra\Mod k,
\]
which takes triangles to exact sequences and coproduct to products. Then, by hypothesis this functor is represented by an object, say $\tau(M)\in\cd'$. By Yoneda lemma it turns out that the map $M\mapsto\tau(M)$ underlies a functor $\cd\ra\cd'$ which is right adjoint to $\iota$. Therefore, by \cite[subsection 1.1]{KellerVossieck88} we have that $\cd'$ is an aisle. Finally, if $\cp$ is a set of objects of $\cd$ which are perfect in $\Tria(\cp)$, then by H.~Krause's theorem \cite[Theorem A]{Krause2002} we know that $\Tria(\cp)$ satisfies the Brown's representability theorem.
\end{proof}

\begin{remark}
By using Lemma \ref{Generators and infinite devissage} and the lemma above we have the following: if a triangulated category $\cd$ with small coproducts is generated by a set $\cp$ of objects such that $\Tria(\cp)$ is an aisle in $\cd$ (\eg if the objects of $\cp$ are perfect in $\Tria(\cp)$), then it satisfies the principle of infinite d\'{e}vissage with respect to that set, \ie $\cd=\Tria(\cp)$.
\end{remark}

A triangulated category with small coproducts is \emph{perfectly} (respectively, \emph{superperfectly}, \emph{compactly}) \emph{generated} if it is generated by a set of perfect (respectively, superperfect, compact) objects. A TTF triple $(\cx,\cy,\cz)$ on a triangulated category with small coproducts is \emph{perfectly} (respectively, \emph{superperfectly}, \emph{compactly}) \emph{generated} if so is $\cx$ as a triangulated category.

\subsection{Smashing subcategories}\label{Smashing subcategories}

Let $\cd$ be a triangulated category with small coproducts. A subcategory $\cx$ of $\cd$ is \emph{smashing} if it is a full triangulated subcategory of $\cd$ which, moreover, is an aisle in $\cd$ whose associated coaisle $\cx^{\bot}$ is closed under small coproducts. It is proved in \cite{NicolasTesis} that this agrees with probably more standard definitions of ``smashing subcategory''.

Smashing subcategories allow to transfer local phenomena to global phenomena:
\begin{lemma}
If $M$ is a perfect (respectively, superperfect, compact) object of a smashing subcategory $\cx$ of a triangulated category $\cd$ with small coproducts, then $M$ is perfect (respectively, superperfect, compact) in $\cd$.
\end{lemma}
\begin{proof}
Use that a small coproduct of triangles associated to the t-structure $(\cx,\cx^{\bot})$ is again a triangle associated to this t-structure.
\end{proof}

Now we fix the notation for a particular kind of construction which will be crucial at certain steps. Let $\cd$ be a triangulated category and let
\[M_{0}\arr{f_{0}}M_{1}\arr{f_{1}}M_{2}\ra\dots
\]
be a sequence of morphisms of $\cd$ such that the coproduct $\coprod_{n\geq 0}M_{n}$ exists in $\cd$. The \emph{Milnor colimit} of this sequence, denoted by $\Mcolim M_{n} $, is given, up to non-unique isomorphism, by the triangle
\[\coprod_{n\geq 0}M_{n}\arr{\id-\sigma}\coprod_{n\geq 0}M_{n}\arr{\pi} \Mcolim M_{n}\ra\coprod_{n\geq 0}M_{n}[1],
\]
where the morphism $\sigma$ has components
\[M_{n}\arr{f_{n}} M_{n+1}\arr{\can}\coprod_{m\geq 0}M_{m}.
\]
The above triangle is said to be the \emph{Milnor triangle} (\cf \cite{Keller1998b, Milnor1962}) associated to the sequence $f_{n}\ko n\geq 0$. The notion of Milnor colimit has appeared in the literature under the name of \emph{homotopy colimit} (\cf \cite[Definition 2.1]{NeemanBokstedt1993}, \cite[Definition 1.6.4]{Neeman2001}) and \emph{homotopy limit} (\cf \cite[subsection 5.1]{Keller1994a}). However, we think it is better to keep this terminology for the notions appearing in the theory of derivators (\cf \cite{Maltsiniotis2001, Maltsiniotis2005, CisinskiNeeman2005}).

\begin{proposition}
Let $\cd$ be a triangulated category with small coproducts and let $P$ be an object of $\cd$. The following conditions are equivalent:
\begin{enumerate}[1)]
\item $P$ is compact in $\cd$.
\item $P$ satisfies:
\begin{enumerate}[2.1)]
\item $P$ is perfect in $\cd$.
\item $P$ is compact in the full subcategory $\mbox{Sum}(\{P[n]\}_{n\in\Z})$ of $\cd$ formed by small coproducts of shifts of $P$.
\item $\Tria(P)^{\bot}$ is closed under small coproducts.
\end{enumerate}
\item $P$ satisfies:
\begin{enumerate}[3.1)]
\item $P$ is compact in $\Tria(P)$.
\item $\Tria(P)^{\bot}$ is closed under small coproducts.
\end{enumerate}
\item $P$ satisfies:
\begin{enumerate}[4.1)]
\item $P$ is superperfect in $\cd$.
\item $P$ is compact in $\Sum(\{P[n]\}_{n\in\Z})$.
\end{enumerate}
\end{enumerate}
\end{proposition}
\begin{proof}
$2)\Rightarrow 1)$ If $P$ is perfect in $\cd$, by Theorem A of \cite{Krause2002} and the adjoint functor argument (\cf Lemma \ref{(Super)perfectness and compactness}) we know that $\cx=\Tria(P)$ is an aisle in $\cd$. Assumption 2.3) says that $\cx$ is a smashing subcategory. Thanks to Lemma \ref{Smashing subcategories}, it suffices to prove that $P$ is compact in $\cx$. For this, we will use the following facts:
\begin{enumerate}[a)]
\item Every object of $\cx$ is the Milnor colimit $\Mcolim X_{n}$ of a sequence
\[X_{0}\arr{f_{0}}X_{1}\arr{f_{1}}X_{2}\arr{f_{2}}\dots
\]
where $X_{0}$ as well as each mapping cone $\cone(f_{m})$ is in $\mbox{Sum}(\{P[n]\}_{n\in\Z})$.
\item Thanks to the proof of Theorem A of \cite{Krause2002} (\cf also the proof of \cite[Theorem 2.2]{Souto2004}), if $\{X_{n}, f_{n}\}_{n\geq 0}$ is a direct system as in a) we know that the natural morphism
\[\lid\cd(P,X_{n})\ra\cd(P,\Mcolim X_{n})
\]
is an isomorphism.
\item Hypothesis 2.2) implies that, for any fixed natural number $m\geq 0$, the functor $\cd(P,?)$ preserves small coproducts of $m$-fold extensions of objects of the class $\mbox{Sum}(\{P[n]\}_{n\in\Z})$.
\end{enumerate}
Let $\Mcolim X^{i}_{n}\ko i\in I$, be an arbitrary family of objects of $\cx$. Then the natural morphism
\[\coprod_{i\in I}\cx(P,\Mcolim X^{i}_{n})\ra\cx(P,\coprod_{i\in I}\Mcolim X^{i}_{n})
\]
is the composition of the following natural isomorphisms
\begin{align}
\coprod_{i\in I}\cx(P,\Mcolim X^{i}_{n})\cong\coprod_{i\in I}\lid_{n\geq 0}\cx(P,X^{i}_{n})\cong\lid_{n\geq 0}\coprod_{i\in I}\cx(P,X^{i}_{n})\cong \nonumber \\
\cong\lid_{n\geq 0}\cx(P,\coprod_{i\in I}X^{i}_{n})\cong \cx(P,\Mcolim \coprod_{i\in I}X^{i}_{n})\cong\cx(P,\coprod_{i\in I}\Mcolim X^{i}_{n}) \nonumber
\end{align}
Hence, $P$ is compact in the smashing subcategory $\cx$.

$3)\Rightarrow 1)$ By the adjoint functor argument we know that $\Tria(P)$ is an aisle in $\cd$, and condition 3.2) ensures that, moreover, it is a smashing subcategory. Hence the lemma above finishes the proof.

$4)\Rightarrow 1)$ Of course, 4) implies 2), and 2) implies 1). However, there exists a shorter proof pointed out by B.~Keller. Let $M_{i}\ko i\in I$, be a family of objects of $\cd$ and take, for each index $i\in I$, an object $Q_{i}\in\Sum(\{P[n]\}_{n\in\Z})$ together with a morphism $Q_{i}\ra M_{i}$ such that the induced morphism
\[\cd(P,Q_{i})\ra\cd(P,M_{i})
\]
is surjective. Since $P$ is superperfect, this implies that the morphism
\[\cd(P,\coprod_{I}Q_{i})\ra\cd(P,\coprod_{I}M_{i}).
\]
is surjective. Now consider the commutative square
\[\xymatrix{\cd(P,\coprod_{I}Q_{i})\ar[r] &\cd(P,\coprod_{I}M_{i}) \\
\coprod_{I}\cd(P,Q_{i})\ar[u]_{\can}^{\wr}\ar[r] &\coprod_{I}\cd(P,M_{i})\ar[u]_{\can}
}
\]
The first vertical arrow is an isomorphism by assumption 4.2, and the horizontal arrows are surjections. Then, the second vertical arrow is surjective. But, of course, it is also injective, and so bijective.
\end{proof}

The following result is a consequence of Lemma \ref{Generators and infinite devissage} and Lemma \ref{(Super)perfectness and compactness}.

\begin{corollary}
If $\cd$ is a perfectly generated triangulated category, then smashing subcategories of $\cd$ are in bijection with TTF triples on $\cd$ via the map
\[\cx\mapsto(\cx,\cx^{\bot},(\cx^{\bot})^{\bot}).
\]
\end{corollary}
\begin{proof}
Indeed, if $(\cx,\cy,\cz)$ is a TTF triple, then $\cx$ is a smashing subcategory since $\cy$ being an aisle is always closed under coproducts. Conversely, if $\cx$ is a smashing subcategory, then $(\cx,\cy)$ is a t-structure on $\cd$. But now, by using Lemma \ref{Generators and infinite devissage} and Lemma \ref{(Super)perfectness and compactness} we have that $\tau^{\cy}$ takes the set of perfect generators of $\cd$ to a set of perfect generators $\cy$. Therefore, $\cy$ is a perfectly generated triangulated category closed under small coproducts in $\cd$, and by the adjoint functor argument we conclude that $\cy$ is an aisle.
\end{proof}

\subsection{B. Keller's Morita theory for derived categories}\label{B. Keller's Morita theory for derived categories}

Let $\ca$ be a small dg category (\cf \cite{Keller1994a, Keller2006b}). It was proved by B.~Keller \cite{Keller1994a} that its derived category $\cd\ca$ is a triangulated category compactly generated by the modules $A^{\we}:=\ca(?,A)$ represented by the objects $A$ of $\ca$. Conversely, he also proved \cite[Theorem 4.3]{Keller1994a} that every \emph{algebraic triangulated category} (namely, a triangulated category which is triangle equivalent to the stable category of a Frobenius category \cite{Heller60, Happel87, KellerVossieck87, GelfandManin, Keller1996}) with small coproducts and with a set $\cp$ of compact generators is the derived category of a certain dg category whose set of objects is equipotent to $\cp$.

The proof of Theorem 4.3 of \cite{Keller1994a} has two parts.

\emph{First part:} it is proved that every algebraic triangulated category admits an enhancement, \ie comes from an exact dg category. We say that a dg category $\ca'$ is \emph{exact} or \emph{pretriangulated} \cite{Keller1999, Keller2006b} if the image of the (fully faithful) Yoneda functor
\[\Zy 0\ca'\ra \cc\ca'\ko M\mapsto M^{\we}:=\ca'(?,M)
\]
is stable under shifts and extensions (in the sense of the exact structure on $\cc\ca$ in which the conflations are the degreewise split short exact sequences). If $\ca'$ is an exact dg category, then $\Zy 0\ca'$ becomes a Frobenius category and $\ul{\Zy 0\ca'}=\H 0\ca'$ is a full triangulated subcategory of $\ch\ca'$. B.~Keller has shown \cite[Example 2.2.c)]{Keller1999} that if $\cc$ is a Frobenius category with class of conflations $\ce$, then $\ul{\cc}=\H 0\ca'$ for the exact dg category $\ca'$ formed by the acyclic complexes with $\ce$-projective-injective components over $\cc$.

\emph{Second part:} it proves the following.

\begin{proposition}
Let $\ca'$ be an exact dg category such that the associated triangulated category $\H 0\ca'$ is compactly generated by a set $\cb$ of objects. Consider $\cb$ as a dg category, regarded as a full subcategory of $\ca'$. Then, the map
\[M\mapsto M^{\we}_{\ \ |_{\cb}}:=\ca'(?,M)_{|_{\cb}}
\]
induces a triangle equivalence
\[\H 0\ca'\arr{\sim}\cd\cb.
\]
\end{proposition}

The dg category associated to the Frobenius category in the first part of the proof of \cite[Theorem 4.3]{Keller1994a} is not very explicit. However, many times in practice we are already like in the second step of the proof, which allows us a better choice of the dg category. In what follows, we will recall how this better choice can be made.

Let $\cp$ be a set of objects of $\cd\ca$ and define $\cb$ as the dg subcategory of the exact dg category $\cc_{dg}\ca$ (\cf \cite{Keller2006b} for the notation) formed by the $\ch$\emph{-injective} or \emph{fibrant resolutions} $\textbf{i}P$ \cite{Keller1994a, Keller2006b} of the modules $P$ of $\cp$. Then we have a dg $\cb$-$\ca$-bimodule $X$ defined by
\[X(A,B):=B(A)
\]
for $A$ in $\ca$ and for $B$ in $\cb$, and we have a pair $(?\otimes_{\cb}X,\ch om_{\ca}(X,?))$ of adjoint dg functors
\[\xymatrix{\cc_{dg}\ca\ar@<1ex>[d]^{\ch om_{\ca}(X,?)} \\
\cc_{dg}\cb\ar@<1ex>[u]^{?\otimes_{\cb}X}
}
\]
For instance, $\ch om_{\ca}(X,?)$ is defined by $\ch om_{\ca}(X,M):=(\cc_{dg}\ca)(?,M)_{|_{\cb}}$ for $M$ in $\cc_{dg}\ca$. These functors induce a pair of adjoint triangle functors between the corresponding \emph{categories up to homotopy} \cite{Keller1994a, Keller2006b}
\[\xymatrix{\ch\ca\ar@<1ex>[d]^{\ch om_{\ca}(X,?)} \\
\ch\cb\ar@<1ex>[u]^{?\otimes_{\cb}X}.
}
\]
The \emph{total right derived functor} $\textbf{R}\ch om_{\ca}(X,?)$ is the composition
\[\cd\ca\arr{\bf i}\ch_{\bf i}\ca\hookrightarrow\ch\ca\arr{\ch om_{\ca}(X,?)}\ch\cb\ra\cd\cb,
\]
where $\textbf{i}$ is the $\ch$-\emph{injective resolution} or \emph{fibrant resolution} functor \cite{Keller1994a, Keller2006b}, and the \emph{total left derived functor} $?\otimes^{\textbf{L}}_{\cb}X$ is the composition
\[\cd\cb\arr{\bf p}\ch_{\bf p}\ca\hookrightarrow\ch\cb\arr{?\otimes_{\cb}X}\ch\ca\ra\cd\ca,
\]
where $\textbf{p}$ is the $\ch$-\emph{projective resolution} or \emph{cofibrant resolution} functor \cite{Keller1994a, Keller2006b}.
They form a pair of adjoint triangle functors at the level of derived categories
\[\xymatrix{\cd\ca\ar@<1ex>[d]^{\textbf{R}\ch om_{\ca}(X,?)} \\
\cd\cb\ar@<1ex>[u]^{?\otimes_{\cb}^\textbf{L}X}
}
\]

The following is an easy consequence of Proposition \ref{B. Keller's Morita theory for derived categories}.

\begin{corollary}
Assume that the objects of $\cp$ are compact in the full triangulated subcategory $\Tria(\cp)$ of $\cd\ca$. Then:
\begin{enumerate}[1)]
\item the functors $(?\otimes_{\cb}^\textbf{L}X,\textbf{R}\ch om_{\ca}(X,?))$ induce mutually quasi-inverse triangle equivalences
\[\xymatrix{\Tria(\cp)\ar@<1ex>[rr]^{\textbf{R}\ch om_{\ca}(X,?)} && \cd\cb\ar@<1ex>[ll]^{?\otimes_{\cb}^\textbf{L}X}
}
\]
which gives a bijection between the objects of $\cp$ and the $\cb$-modules $B^{\we}$ represented by the objects $B$ of $\cb$,
\item $\Tria(\cp)$ is an aisle in $\cd\ca$ with truncation functor given by the map
\[M\mapsto \textbf{R}\ch om_{\ca}(X,M)\otimes^{\bf L}_{\cb}X.
\]
\end{enumerate}
\end{corollary}
\begin{proof}
1) Since the $\ch$-injective resolution functor $\bf i:\cd\ca\arr{\sim}\ch_{\bf i}\ca$ is a triangle equivalence, it induces a triangle equivalence between $\Tria(\cp)$ and a certain full triangulated subcategory of $\ch_{\bf i}\ca$. This subcategory is algebraic, \ie the stable category $\ul{\cc}$ of a certain Frobenius category $\cc$, because it is a subcategory of $\ch\ca$.  Since $\ch\ca$ is the $\H 0$-category of the exact dg category $\cc_{dg}\ca$, then $\ul{\cc}$ is the $\H 0$-category of an exact dg subcategory of $\cc_{dg}\ca$. Moreover,  $\ul{\cc}$ is compactly generated by the objects of $\cb$. Then, by Proposition \ref{B. Keller's Morita theory for derived categories} the restriction of $\ch om_{\ca}(X,?)$ to $\ul{\cc}$ induces a triangle equivalence
\[\ul{\cc}\arr{\ch om_{\ca}(X,?)}\ch\cb\ra\cd\cb.
\]
The picture is:
\[\xymatrix{ & \H 0(\cc_{dg}\ca)=\ch\ca & \\
\cd\ca\ar[r]^{\bf i}_{\sim} & \ch_{\bf i}\ca\ar@{^(->}[u] & \\
\Tria(\cp)\ar[r]_{\sim}\ar@{^(->}[u] & \ul{\cc}\ar[r]_{\sim}\ar@{^(->}[u] & \cd\cb
}
\]
The composition of the bottom arrows is the restriction of $\textbf{R}\ch om_{\ca}(X,?)$ to $\Tria(\cp)$. Notice that by the adjoint functor argument (\cf Lemma \ref{(Super)perfectness and compactness}) $\Tria(\cp)$ is an aisle, and so $\Tria(\cp)=\ ^{\bot}(\Tria(\cp)^{\bot})$. Finally, by adjunction the image of $?\otimes_{\cb}^\textbf{L}X$ is in $\ ^{\bot}(\Tria(\cp)^{\bot})=\Tria(\cp)$. Alternatively, we can use that $\cd\cb$ satisfies the principle of infinite d\'{e}vissage with respect to the modules $B^{\we}\ko B\in\cb$, to deduce that the image of $?\otimes^{\L}_{\cb}X$ is contained in $\Tria(\cp)$. Then, we can prove that
\[?\otimes^{\L}_{\cb}X:\cd\cb\ra\Tria(\cp)
\]
is a triangle equivalence by using \cite[Lemma 4.2]{Keller1994a}.

2) Since $?\otimes^\textbf{L}_{\cb}X:\cd\cb\ra\cd\ca$ is fully faithful, the unit $\eta$ of the adjunction $(?\otimes_{\cb}^\textbf{L}X,\textbf{R}\ch om_{\ca}(X,?))$ is an isomorphism. Then, when we apply the functor $?\otimes_{\cb}^\textbf{L}X\circ \textbf{R}\ch om_{\ca}(X,?)$ to the counit $\delta$ we get an isomorphism. This shows that for each module $M$ in $\cd\ca$, the triangle of $\cd\ca$
\[\textbf{R}\ch om_{\ca}(X,M)\otimes^{\bf L}_{\cb}X\arr{\delta_{M}}M\ra M'\arr{+}
\]
satisfies that $\textbf{R}\ch om_{\ca}(X,M')\otimes^{\bf L}_{\cb}X=0$. Since $?\otimes^{\bf L}_{\cb}X$ is fully faithful, then $\textbf{R}\ch om_{\ca}(X,M')=0$. That is to say,
\[\textbf{R}\ch om_{\ca}(X,M')(P)=(\cc_{dg}\ca)(P,\textbf{i}M')
\]
is acyclic for each object $P$ in $\cp$. But then, we have
\[\H n(\cc_{dg}\ca)(P,\textbf{i}M')=(\ch\ca)(P,\textbf{i}M'[n])\cong(\cd\ca)(P,M'[n])=0
\]
for each object $P$ of $\cp$ and each integer $n\in\Z$. This implies, by infinite d\'{e}vissage, that $M'$ belongs to the coaisle $\Tria(\cp)^{\bot}$ of $\Tria(\cp)$.
\end{proof}

\begin{remark}
This result generalizes \cite[Theorem 1.6]{Jorgensen2006} and \cite[Theorem 2.1]{DwyerGreenlees}. Indeed, if $\ca$ is the dg category associated to the dg $k$-algebra $A$ and the set $\cp$ has only one element $P$, then $\cb$ is the dg category associated to the dg algebra $B:=(\cc_{dg}A)(\textbf{i}P,\textbf{i}P)$, the dg $\cb$-$\ca$-bimodule $X$ corresponds to the dg $B$-$A$-bimodule $\textbf{i}P$ and the triangle equivalence
\[\textbf{R}\ch om_{\ca}(X,?):\Tria(P)\arr{\sim}\cd B
\]
is given by $\textbf{R}\Hom_{A}(\textbf{i}P,?)$.
\end{remark}

\section{Parametrization}\label{Parametrization}

\subsection{Recollement-defining classes}\label{Recollement-defining classes}

A class $\cp$ of objects of a triangulated category $\cd$ is \emph{recollement-defining} if the class $\cy$ of objects which are right orthogonal to all the shifts  of objects of $\cp$ is both an aisle and a coaisle in $\cd$.

Notice that, in this case, one has that the triangulated category $\ ^{\bot}\cy$ is generated by $\cp$ thanks to Lemma \ref{Generators and infinite devissage}.

In the following subsections, we will show how to weaken the conditions imposed to a set in order to be recollement-defining in some particular frameworks.

\subsection{Recollement-defining sets in aisled categories}\label{Recollement-defining sets in aisled categories}

A triangulated category $\cd$ is \emph{aisled} if it has a set of generators, small coproducts and for every set $\cq$ of objects of $\cd$ we have that $\Tria(\cq)$ is an aisle in $\cd$.

\begin{lemma}
Let $\cd$ be an aisled triangulated category. Then, for a set $\cp$ of objects of $\cd$ the following assertions are equivalent:
\begin{enumerate}[1)]
\item $\cp$ is a recollement-defining set.
\item The class $\cy$ of objects of $\cd$ which are right orthogonal to all the shifts of objects of $\cp$ is closed under small coproducts.
\end{enumerate}
In this case $\Tria(\cp)=\ ^{\bot}\cy$.
\end{lemma}
\begin{proof}
$2)\Rightarrow 1)$ Since $\cd$ is aisled, then $\Tria(\cp)$ is an aisle in $\cd$. By infinite d\'{e}vissage, we have that the coaisle is precisely $\cy$. If $\cg$ is a set of generators of $\cd$, then by using Lemma \ref{Generators and infinite devissage} we know that $\tau^{\cy}(\cg)$ is a set of generators of $\cy$. Notice that $\Tria(\tau^{\cy}(\cg))$ is an aisle in $\cd$ contained in $\cy$. Hence, by Lemma \ref{Generators and infinite devissage}, it turns out that $\Tria(\tau^{\cy}(\cg))=\cy.$ This proves that $\cy$ is an aisle in $\cd$.
\end{proof}

In the remainder of the subsection we will show that aisled triangulated categories do exist and that some of the most familiar triangulated categories are among them.

First, notice that by using Corollary 3.12 of \cite{Porta2007} and the adjoint functor argument one can prove that every \emph{well-generated} \cite{Neeman2001} triangulated category (in particular, the derived category of any small dg category) is aisled. Therefore, well-generated triangulated categories form a class of triangulated categories which are aisled by `global' reasons. Let us present a class of triangulated categories which are aisled by `local reasons'. For this we need some terminology.

Let $\cc$ be a Frobenius category with small colimits. For a set $\cq$ of objects of $\cc$ we define
$\ci_{\cq}$ to be the set formed by the inflations $i_{Q}:Q\ra IQ$ where $Q$ runs through $\cq$. We say that $\cq$ is \emph{self-small} if it is closed under shifts (in the stable category $\ul{\cc}$) and its objects are small relative to $\ci_{\cq}$-cell (\cf \cite[Definition 2.1.3]{Hovey1999} for the definition of ``small'' and \cite[Definition 2.1.9]{Hovey1999} for the definition of $\ci_{\cq}$-cell).

The smallness condition is interesting since it allows us to construct aisles.

\begin{theorem}
If $\cc$ is a Frobenius category with small colimits and $\cq$ is a self-small set of objects of $\cc$, then:
\begin{enumerate}[1)]
\item $\Tria(\cq)$ is an aisle in $\ul{\cc}$ with associated coaisle given by the class $\cq^{\bot_{\ul{\cc}}}$.
\item The objects of $\Tria(\cq)$ are precisely those isomorphic to an $\ci_{\cq}$-cell complex.
\end{enumerate}
\end{theorem}

The proof of the above theorem can be found in \cite{Nicolas2007b}. Here we have more examples of aisled triangulated categories.

\begin{example}
Let $\ca$ be a small dg category. Then the category of right dg $\ca$-modules, $\cc\ca$, is a Frobenius category (with conflations given by the degreewise split short exact sequences) such that every set of objects closed under shifts is self-small. To prove it, one can easily generalize the argument of \cite[Lemma 2.3.2]{Hovey1999}. Notice that the category of right dg $\ca$-modules up to homotopy $\ch\ca$ is not aisled since in general it does not admit a set of generators (\cf \cite[Lemma E.3.2]{Neeman2001}). However, for every set $\cq$ of objects of $\ch\ca$ we have that $\Tria_{\ch\ca}(\cq)$ is an aisle in $\ch\ca$ and an aisled triangulated category. In particular, by taking $\cq$ to be the set of representable modules $A^{\we}\ko A\in\ca$, we deduce that $\cd\ca$ is aisled.
\end{example}

\begin{remark}
It is worth noting that, thanks to Lemma \ref{Generators and infinite devissage}, if $\cd$ is an aisled triangulated category and $\cd'$ is a full triangulated subcategory of $\cd$ closed under small coproducts and generated by a set of objects $\cq$, then $\cd'=\Tria(\cq)$.
\end{remark}

\subsection{Recollement-defining sets in perfectly generated triangulated categories}\label{Recollement-defining sets in perfectly generated triangulated categories}

\begin{lemma}
Let $\cd$ be a perfectly generated triangulated category, $\cp$ a set of objects of $\cd$ and $\cy$ the class of objects of $\cd$ which are right orthogonal to all the shifts of objects of $\cp$. The following assertions are equivalent:
\begin{enumerate}[1)]
\item $\cp$ is recollement-defining.
\item $\cy$ is a coaisle in $\cd$ closed under small coproducts.
\end{enumerate}
If $\cp$ consists of perfect objects, the above statements are also equivalent to:
\begin{enumerate}[3)]
\item $\cy$ is closed under small coproducts.
\end{enumerate}
In this last case $\Tria(\cp)=\ ^{\bot}\cy$.
\end{lemma}
\begin{proof}
$2)\Rightarrow 1)$ $\ ^{\bot}\cy$ is a smashing subcategory in a perfectly generated triangulated category, which implies that its associated coaisle $\cy$ is also an aisle (\cf subsection \ref{Smashing subcategories}).

$3)\Rightarrow 2)$ It is clear that $\Tria(\cp)$ is an aisle in $\cd$ (use the adjoint functor argument, \cf Lemma \ref{(Super)perfectness and compactness}) whose associated coaisle is $\cy$ by infinite d\'{e}vissage.
\end{proof}

Notice that any set of superperfect (\eg compact) objects of a perfectly ge\-ne\-ra\-ted triangulated category satisfies condition 3) of the above lemma, and so it is recollement-defining.

\subsection{Parametrization of TTF triples on triangulated categories}\label{Parametrization of TTF triples on triangulated categories}

\begin{proposition}
Let $\cd$ be a triangulated category with a set of generators. Consider the map which takes a set $\cp$ of objects of $\cd$ to the triple
\[(\ ^{\bot}\cy,\cy,\cy^{\bot})
\]
of subcategories of $\cd$, where $\cy$ is formed by those objects which are right orthogonal to all the shifts of objects of $\cp$. The following assertions hold:
\begin{enumerate}[1)]
\item This map defines a surjection from the class of all recollement-defining sets onto the class of all the TTF triples on $\cd$.
\item If $\cd$ is aisled, then this map induces a surjection from the class of objects $P$ such that $\{P[n]\}_{n\in\Z}^{\bot}$ is closed under small coproducts onto the class of all TTF triples on $\cd$.
\item If $\cd$ is perfectly generated, then this map induces surjections from
\begin{enumerate}[3.1)]
\item the class of perfect objects $P$ such that $\{P[n]\}_{n\in\Z}^{\bot}$ is closed under small coproducts onto the class of all perfectly generated TTF triples.
\item the class of superperfect objects onto the class of superperfectly generated TTF triples.
\item the class of sets of compact objects onto the class of compactly generated TTF triples.
\end{enumerate}
\end{enumerate}
\end{proposition}
\begin{proof}
1) Let $(\cx,\cy,\cz)$ be a TTF triple on $\cd$ and let $\cg$ be a set of generators of $\cd$. Since $\tau^{\cz}(\cg)$ is a set of generators of $\cz$ (\cf Lemma \ref{Generators and infinite devissage}), and the composition $\cz\arr{z}\cd\arr{\tau_{\cx}}\cx$ is a triangle equivalence, then $\tau_{\cx}z\tau^{\cz}(\cg)$ is a set of generators of $\cx$. Now by using Lemma \ref{Generators and infinite devissage} we know that $\cy$ is the set of objects which are right orthogonal to all the shifts of objects of $\tau_{\cx}z\tau^{\cz}(\cg)$, which proves simultaneously that $\tau_{\cx}z\tau^{\cz}(\cg)$ is a recollement-defining set and that the TTF triple comes from a recollement-defining set.

2) We use Lemma \ref{Recollement-defining sets in aisled categories} and the fact that if $\cp$ is a recollement-defining set, then $\{\coprod_{P\in \cp}P\}$ is also a recollement-defining set which is sent to the same TTF triple onto which $\cp$ was sent.

For 3.1) and 3.2) we use the idea of the proof of 2) together with the fact that the class of (super)perfect objects is closed under small coproducts. For 3.3) we use Lemma \ref{Smashing subcategories}.
\end{proof}

Recall that an object $M$ of a triangulated category $\cd$ is called \emph{exceptional} if it has no self-extensions, \ie $\cd(M,M[n])=0$ for each integer $n\neq 0$.

The following corollary generalizes \cite[Theorem 3.3]{Jorgensen2006}, and also \cite[Theorem 2.16]{Heider2007} via \cite[Theorem 2.9]{Heider2007}.

\begin{corollary}
Let $\cd$ be a triangulated category which is either perfectly generated or aisled. The following assertions are equivalent:
\begin{enumerate}[1)]
\item $\cd$ is a recollement of triangulated categories generated by a single compact (and exceptional) object.
\item There are (exceptional) objects $P$ and $Q$ of $\cd$ such that:
\begin{enumerate}[2.1)]
\item $P$ is compact.
\item $Q$ is compact in $\Tria(Q)$.
\item $\cd(P[n],Q)=0$ for each $n\in\Z$.
\item $\{P\ko Q\}$ generates $\cd$.
\end{enumerate}
\item There is a compact (and exceptional) object $P$ such that $\Tria(P)^{\bot}$ is generated by a compact (and exceptional) object in $\Tria(P)^{\bot}$.
\end{enumerate}
In case $\cd$ is compactly generated by a single object the former assertions are e\-qui\-va\-lent to:
\begin{enumerate}[4)]
\item There is a compact (and exceptional) object $P$ (such that $\Tria(P)^{\bot}$ is generated by an exceptional compact object).
\end{enumerate}
In case $\cd$ is algebraic the former assertions are equivalent to:
\begin{enumerate}[5)]
\item $\cd$ is a recollement of derived categories of dg algebras (concentrated in degree $0$, \ie ordinary algebras).
\end{enumerate}
\end{corollary}
\begin{proof}
$1)\Rightarrow 2)$ If $(\cx,\cy,\cz)$ is the TTF triple corresponding to the recollement of 1), we can take $P$ to be a compact generator of $\cx$ and $Q$ to be a compact generator of $\cy$. Of course, conditions 2.2), 2.3) and 2.4) are satisfied. Finally, by Lemma \ref{Smashing subcategories}, $P$ is compact in $\cd$.

$2)\Rightarrow 1)$ Since $P$ is compact then $\Tria(P)$ is a smashing subcategory of $\cd$. Conditions 2.3) and 2.4) say that $Q$ generates $\Tria(P)^{\bot}$. Moreover, $\Tria(Q)$ is contained in $\Tria(P)^{\bot}$ and condition 2.2) ensures that $\Tria(Q)$ is an aisle in $\cd$. Therefore Lemma \ref{Generators and infinite devissage} implies that $\Tria(P)^{\bot}=\Tria(Q)$, and so $Q$ is a compact generator of $\Tria(P)^{\bot}$. Finally, by using either that $\cd$ is perfectly generated (together with Lemma \ref{Recollement-defining sets in perfectly generated triangulated categories}) or that $\cd$ is aisled (together with Lemma \ref{Recollement-defining sets in aisled categories}) we have that $(\Tria(P), \Tria(P)^{\bot},(\Tria(P)^{\bot})^{\bot})$ is a TTF triple.

$2)\Leftrightarrow 3)$ is clear, and $3)\Leftrightarrow 4)$ is also clear thanks to Lemma \ref{Generators and infinite devissage} and Lemma \ref{(Super)perfectness and compactness}.

$1)\Rightarrow 5)$ We use that, by \cite[Theorem 4.3]{Keller1994a}, an algebraic triangulated category compactly generated by a single object is triangle equivalent to the derived category of a dg algebra (\cf subsection \ref{B. Keller's Morita theory for derived categories}).

To deal with the case of exceptional compact objects, one uses that an algebraic triangulated category compactly generated by an exceptional object is triangle equivalent to the derived category of an ordinary algebra (\cf for instance \cite[Theorem 8.3.3]{Keller1998b}).
\end{proof}

Thanks to Corollary \ref{B. Keller's Morita theory for derived categories}, the dg algebras announced in 5) can be chosen to be particularly nice in case $\cd$ is the derived category $\cd A$ of a dg algebra $A$. Indeed, if $P$ and $Q$ are like in 2) and $(\cx,\cy,\cz)=(\Tria(P),\Tria(Q),\Tria(Q)^{\bot})$, then the picture is
\[\xymatrix{\cd B\ar@<1ex>[rr]^{?\otimes^{\bf L}_{B}\textbf{i}Q} && \Tria(Q)\ar@<1ex>[ll]^{\textbf{R}\Hom_{A}(\textbf{i}Q,?)}\ar[r]^{y} & \cd A\ar@/_1pc/[l]_{\tau^{\cy}}\ar@/_-1pc/[l]^{\tau_{\cy}}\ar[r]^{\tau_{\cx}} & \Tria(P)\ar@/_-1pc/[l]^{x}\ar@/_1pc/[l]_{z\tau^{\cz}x}\ar@<1ex>[rr]^{\textbf{R}\Hom_{A}(\textbf{i}P,?)} && \cd C\ar@<1ex>[ll]^{?\otimes^{\bf L}_{C}\textbf{i}P}
}
\]
where $B$ is the dg algebra $(\cc_{dg}A)(\textbf{i}Q,\textbf{i}Q)$ and $C$ is the dg algebra $(\cc_{dg}A)(\textbf{i}P,\textbf{i}P)$.

\begin{example}\cite[Example 8]{Konig1991}
Let $k$ be a field and let $A=k(Q,R)$ be the finite dimensional $k$-algebra associated to the quiver
\[\xymatrix{Q=(1\ar@<1ex>[r]^{\alpha} & 2\ar@<1ex>[l]^{\beta})
}
\]
with relations $R=(\alpha\beta\alpha)$. One easily check that $P:=P_{2}=e_{2}A$ and $Q:=S_{1}=e_{1}A/e_{1}\op{rad}(A)$ are exceptional objects of $\cd A$ and satisfy conditions 2.1)--2.4) of the above corollary. Now, since the dg algebra $(\cc_{dg}A)(\textbf{i}P_{2},\textbf{i}P_{2})$ has cohomology concentrated in degree $0$ and isomorphic, as an algebra, to $C:=\op{End}_{A}(P_{2})$, its derived category is triangle equivalent to $\cd C$. Similarly, the derived category of $(\cc_{dg}A)(\textbf{i}S_{1},\textbf{i}S_{1})$ is triangle equivalent to the derived category of $B:=\op{End}_{A}(S_{1})$. By applying Corollary \ref{Parametrization of TTF triples on triangulated categories}, we know that there exists a recollement
\[\xymatrix{\cd B\ar[r]^{i_{*}} & \cd A\ar@/_1pc/[l]\ar@/_-1pc/[l]\ar[r] & \cd C\ar@/_1pc/[l]\ar@/_-1pc/[l]^{j_{!}}
}
\]
with $i_{*}B=S_{1}$ and $j_{!}C=P_{2}$.
\end{example}

\section{Homological epimorphisms of dg categories}\label{Homological epimorphisms of dg categories}

Let $F:\ca\ra\cb$ be a dg functor between dg categories and suppose the co\-rres\-pon\-ding restriction along $F$
\[F_{*}:\cd\cb\ra\cd\ca
\]
is fully faithful. Let $U$ be the $\ca$-$\cb$-bimodule defined by $U(B,A):=\cb(B,FA)$ and let $V$ be the $\cb$-$\ca$-bimodule defined by $V(A,B):=\cb(FA,B)$. Since the functor $F_{*}$ admits a left adjoint
\[?\otimes_{\ca}^{\bf L}U:\cd\ca\ra\cd\cb
\]
and a right adjoint
\[\textbf{R}\ch om_{\ca}(V,?):\cd\ca\ra\cd\cb,
\]
then the essential image $\cy$ of $F_{*}$ is a full triangulated subcategory which is both an aisle and a coaisle in $\cd\ca$. This shows that there exists a TTF triple on $\cd\ca$ whose central class $\cy$ is triangle equivalent to $\cd\cb$.

In fact, by using Lemma \ref{Generators and infinite devissage} and B.~Keller's Morita theory for derived categories (\cf subsection \ref{B. Keller's Morita theory for derived categories}), it is clear that the central class of any TTF triple on the derived category $\cd\ca$ of a dg category is always triangle equivalent to the derived category $\cd\cb$ of a certain dg category. We will prove in this section that, up to replacing $\ca$ by a quasi-equivalent \cite{Tabuada2005a} dg category, the new dg category $\cb$ can be chosen so as to be linked to $\ca$ by a dg functor $F:\ca\ra\cb$ whose corresponding restriction $F_{*}:\cd\cb\ra\cd\ca$ is fully faithful.

Let us show first a nice characterization of this kind of morphisms. For this, notice that a morphism $F:\ca\ra\cb$ of dg categories also induces a restriction of the form $\cd(\cb^{op}\otimes_{k}\cb)\ra\cd(\ca^{op}\otimes_{k}\ca)$, still denoted by $F_{*}$, and that $F$ can be viewed as a morphism $F:\ca\ra F_{*}\cb$ in $\cd(\ca^{op}\otimes_{k}\ca)$. Let $X\ra\ca\arr{F}F_{*}\cb\ra X[1]$ be the triangle of $\cd(\ca^{op}\otimes_{k}\ca)$ induced by $F$.

\begin{lemma}
Let $F:\ca\ra\cb$ be a dg functor between small dg categories. The following statements are equivalent:
\begin{enumerate}[1)]
\item $F_{*}:\cd\cb\ra\cd\ca$ is fully faithful.
\item The counit $\delta$ of the adjunction $(?\otimes^{\bf L}_{\ca}U,F_{*})$ is an isomorphism.
\item The counit $\delta_{B^{\we}}:F_{*}(B^{\we})\otimes^{\bf L}_{\ca}U\ra B^{\we}$ is an isomorphism for each object $B$ of $\cb$.
\item $F$ satisfies the following:
\begin{enumerate}[4.1)]
\item The modules $(FA)^{\we}\ko A\in\ca$, form a set of compact generators of $\cd\cb$.
\item $X(?,A)\otimes^{\bf L}_{\ca}U=0$ for each object $A$ of $\ca$.
\end{enumerate}
\item $F$ satisfies the following:
\begin{enumerate}[5.1)]
\item The modules $(FA)^{\we}\ko A\in\ca$, form a set of compact generators of $\cd\cb$.
\item The class $\cy$ of modules $M\in\cd\ca$ such that $(\cd\ca)(X(?,A)[n],M)=0$, for each $A\in\ca$ and $n\in\Z$, is closed under small coproducts and $(F_{*}\cb)(?,A)\in\cy$ for each $A\in\ca$.
\end{enumerate}
\end{enumerate}
\end{lemma}
\begin{proof}
The equivalence $1)\Leftrightarrow 2)$ is a general fact about adjoint functors, and implication $2)\Rightarrow 3)$ is clear.

$3)\Rightarrow 2)$ follows from the implications $3)\Rightarrow 4)\Rightarrow 2)$ below. However, there is a shorter and more natural proof: one can use that $\cd\cb$ satisfies the principle of infinite d\'{e}vissage with respect to the $B^{\we}\ko B\in\cb$, and that the modules $M$ with invertible $\delta_{M}$ form a strictly full triangulated subcategory of $\cd\cb$ closed under small coproducts and containing the $B^{\we}\ko B\in\cb$.

$3)\Rightarrow 4)$ It is easy to show that if $F_{*}$ is fully faithful, then the objects $(FA)^{\we}\ko A\in\ca$ form a set of (compact) generators. Now, for an object $A\in\ca$ we get the triangle in $\cd\ca$
\[X(?,A)\ra A^{\we}\arr{F}F_{*}((FA)^{\we})\ra X(?,A)[1].
\]
If we apply $?\otimes^{\bf L}_{\ca}U$ then we get the triangle
\[X(?,A)\otimes^{\bf L}_{\ca}U\ra (FA)^{\we}\ra F_{*}((FA)^{\we})\otimes^{\bf L}_{\ca}U\ra X(?,A)\otimes^{\bf L}_{\ca}U[1]
\]
of $\cd\cb$, which gives the triangle
\[F_{*}(X(?,A)\otimes^{\bf L}_{\ca}U)\ra F_{*}((FA)^{\we})\ra F_{*}(F_{*}((FA)^{\we})\otimes^{\bf L}_{\ca}U)\ra F_{*}(X(?,A)\otimes^{\bf L}_{\ca}U)[1]
\]
of $\cd\ca$. The morphism $F_{*}((FA)^{\we})\ra F_{*}(F_{*}((FA)^{\we})\otimes^{\bf L}_{\ca}U)$ is induced by the unit of the adjunction $(?\otimes^{\bf L}_{\ca}U,F_{*})$, it is the right inverse of $F_{*}(\delta_{(FA)^{\we}})$, and it is an isomorphism if and only if $F_{*}(X(?,A)\otimes^{\bf L}_{\ca}U)=0$. Since $F_{*}$ is fully faithful, then it reflects both isomorphism and zero objects. Hence, we have that $\delta_{(FA)^{\we}}$ is an isomorphism for every $A\in\ca$ if and only if condition 4.2) holds.

$4)\Rightarrow 2)$ Condition 4.1) implies that $F_{*}$ reflects both zero objects and isomorphisms. By using the same argument as in $3)\Rightarrow 4)$, one proves that condition 4.2) implies that $\delta_{(FA)^{\we}}$ is an isomorphism for each $A\in\ca$. Then condition 4.1) guarantees that we can use infinite d\'{e}vissage to prove that 2) holds.

$4), 1)\Rightarrow 5)$ We know that the essential image $\cy$ of $F_{*}$ is the middle class of a TTF triple $(\cx,\cy,\cz)$ on $\cd\ca$. Notice that $?\otimes^{\bf L}_{\ca}U$ is left adjoint to $F_{*}$ and so it `is' the truncation functor $\tau^{\cy}$ associated to $\cy$ regarded as a coaisle. Since $X(?,A)\otimes^{\bf L}_{\ca}U=0$ for each $A\in\ca$, then $\Tria(\{X(?,A)\}_{A\in\ca})\subseteq\cx$. Since $A^{\we}\otimes^{\bf L}_{\ca}X=X(?,A)$ is in $\cx$, by infinite d\'{e}vissage we have that the essential image of $?\otimes^{\bf L}_{\ca}X$ is in $\cx$. Hence, the triangle
\[X\ra\ca\arr{F}F_{*}\cb\ra X[1]
\]
of $\cd(\ca^{op}\otimes_{k}\ca)$ induces for each $M\in\cd\ca$ a triangle
\[M\otimes^{\bf L}_{\ca}X\ra M\ra F_{*}(M\otimes^{\bf L}_{\ca}U)\ra (M\otimes^{\bf L}_{\ca}X)[1]
\]
of $\cd\ca$ with $M\otimes^{\bf L}_{\ca}X\in\cx$ and $F_{*}(M\otimes^{\bf L}_{\ca}U)\in\cy$. This proves that $\tau_{\cx}(?)=?\otimes^{\bf L}_{\ca}X$, and thus $\Tria(\{X(?,A)\}_{A\in\ca})=\cx$.

$5)\Rightarrow 4)$ We want to prove
\[(\cd\cb)(X(?,A)\otimes^{\bf L}_{\ca}U,X(?,A)\otimes^{\bf L}_{\ca}U)=0\ko
\]
for each $A\in\ca$, that is to say
\[(\cd\ca)(X(?,A),F_{*}(X(?,A)\otimes^{\bf L}_{\ca}U))=0
\]
for each $A\in\ca$ or, equivalently, $F_{*}(X(?,A)\otimes^{\bf L}_{\ca}U)\in\cy$ for each $A\in\ca$. But in fact, $F_{*}(M\otimes^{\bf L}_{\ca}U)\in\cy$ for every $M\in\cd\ca$, as can be proved by infinite d\'{e}vissage since $\cy$ is closed under small coproducts, $F_{*}$ and $?\otimes^{\bf L}_{\ca}U$ preserve small coproducts and $F_{*}(A^{\we}\otimes^{\bf L}_{\ca}U)=F_{*}\cb(?,A)\in\cy$ for each $A\in\ca$.
\end{proof}

A dg functor $F:\ca\ra\cb$ is a \emph{homological epimorphism} if it satisfies the conditions of the above lemma. From the proof of this lemma it is clear that the recollement associated to the TTF induced by $F$ is of the form
\[\xymatrix{\cd\cb\ar[rr]^{F_{*}} && \cd\ca\ar@/_1pc/[ll]_{?\otimes^{\bf L}_{\ca}U}\ar@/_-1pc/[ll]^{\textbf{R}\ch om_{\ca}(V,?)}\ar[rr]^{\tau_{\cx}} && \cx\ar@/_1pc/[ll]\ar@/_-1pc/[ll]^x
}
\]
where $x$ is the inclusion functor, $\tau_{\cx}(?)=?\otimes^{\bf L}_{\ca}X$ and $\cx=\Tria(\{X(?,A)\}_{A\in\ca})$.

\begin{remark}
Our notion of ``homological epimorphism of dg categories'' is a ge\-ne\-ra\-li\-za\-tion of the notion of ``homological epimorphism of algebras'' due to W. Geigle and H. Lenzing \cite{GeigleLenzing}. Indeed, a morphism of algebras $f:A\ra B$ is a \emph{homological epimorphism} if it satisfies:
\begin{enumerate}[1)]
\item the multiplication $B\otimes_{A}B\ra B$ is bijective,
\item $\op{Tor}^{A}_{i}(B_{A},\ _{A}B)=0$ for every $i\geq 1$.
\end{enumerate}
But this is equivalent to require that the `multiplication' $B\otimes_{A}^{\bf L}B\ra B$ is an isomorphism in $\cd B$, which is precisely condition 3) of the above lemma.
Hence, our lemma recovers and adds some handy characterizations of homological epimorphisms of algebras. Recently, D. Pauksztello \cite{Pauksztello2007} has studied homological epimorphism of dg algebras.
\end{remark}

The following are particular cases of homological epimorphisms:

\begin{example}
Let $I$ be a two-sided ideal of an algebra $A$. The following statements are equivalent:
\begin{enumerate}[1)]
\item The canonical projection $A\ra A/I$ is a homological epimorphism.
\item $\op{Tor}^{A}_{i}(I, A/I)=0$ for every $i\geq 0$.
\item The class of complexes $Y\in\cd A$ such that $(\cd A)(I[n],Y)=0$ for every $n\in\Z$ is closed under small coproducts and $\Ext^{i}_{A}(I,A/I)=0$ for every $i\geq 0$.
\item The class of complexes $Y\in\cd A$ such that $(\cd A)(I[n],Y)=0$ for every $n\in\Z$ is closed under small coproducts and $\Ext^{i}_{A}(A/I,A/I)=0$ for every $i\geq 1$.
\end{enumerate}
The first part of conditions 3) and 4) are always satisfied if $I_{A}$ is compact in $\cd A$, \ie quasi-isomorphic to a bounded complex of finitely generated projective $A$-modules. Note that condition 2) is precisely condition 4) of the above lemma. Also, notice that from $\op{Tor}^{A}_{0}(I,A/I)=0$ follows that $I$ is idempotent. Conversely, if $I$ is idempotent and projective as a right $A$-module (\eg $I=A(1-e)A$ where $e\in A$ is an idempotent such that $eA(1-e)=0$), then condition 2) is clearly satisfied.
\end{example}

The above example contains the unbounded versions of the recollements of Co\-rollary 11, Corollary 12 and Corollary 15 of \cite{Konig1991}. In a forthcoming paper we will prove that all of them restrict to give the recollements of right bounded derived categories of \cite{Konig1991}.

\begin{example}
Let $j:A\ra B$ be an injective morphism of algebras, which we view as an inclusion. The following statements are equivalent:
\begin{enumerate}[1)]
\item $j$ is a homological epimorphism.
\item $\op{Tor}^{A}_{i}(B/A,B)=0$ for every $i\geq 0$.
\item The class of complexes $Y\in\cd A$ such that $(\cd A)((B/A)[n],Y)=0$ for every $n\in\Z$ is closed under small coproducts and $\Ext^i_{A}(B/A,B)=0$ for every $i\geq 0$.
\end{enumerate}
\end{example}

Recall that a dg category $\ca$ is \emph{k-flat} if the functor $?\otimes_{k}\ca(A,A'):\cc k\ra\cc k$ preserves acyclic complexes of $k$-modules for every $A, A'\in\ca$. Of course, this is always the case if $k$ is a field.

\begin{theorem}
Let $\ca$ be a $k$-flat dg category. For every TTF triple $(\cx,\cy,\cz)$ on $\cd\ca$ there exists a homological epimorphism $F:\ca\ra\cb$, bijective on objects, such that the essential image of the restriction of scalars functor $F_{*}:\cd\cb\ra\cd\ca$ is $\cy$.
\end{theorem}
\begin{proof}
For each $A\in\ca$ we consider a fixed triangle
\[X_{A}\ra A^{\we}\arr{\varphi_{A}} Y_{A}\ra X[1]
\]
in $\cd\ca$ with $X_{A}\in\cx$ and $Y_{A}\in\cy$. Assume that each $Y_{A}$ is $\ch$-injective. Let $\cc$ be the dg category given by the full subcategory of $\cc_{dg}\ca$ formed by the objects $Y_{A}\ko A\in\ca$. Clearly, these objects define a $\cc$-$\ca$--bimodule $Y$ as follows:
\[(\cc^{op}\otimes_{k}\ca)^{op}\ra\cc_{dg}k\ko (Y_{A'},A)\mapsto Y(A,Y_{A'}):=Y_{A'}(A),
\]
and the morphisms $\varphi_{A}$ induce a morphism of right dg $\ca$-modules
\[\varphi_{A}: A^{\we}\ra Y(?,Y_{A}).
\]
Let $\xi:Y\ra Y'$ be an $\ch$-injective resolution of $Y$ in $\ch(\cc^{op}\otimes_{k}\ca)$, and let $\cb'$ be the dg category given by the full subcategory of $\cc_{dg}(\cc^{op})$ formed by the objects $Y'(A,?), A\in\ca$. Consider the functor
\[\rho:\ca^{op}\ra\cb',
\]
which takes the object $A$ to $Y'(A,?)$ and the morphism $f\in\ca^{op}(A,A')$ to the morphism $\rho(f)$ defined by
\[\rho(f)(C):Y'(A,C)\ra Y'(A',C)\ko y\mapsto (-1)^{|f||y|}Y'(f,C)(y)
\]
for a homogeneous $f$ of degree $|f|$ and a homogeneous $y$ of degree $|y|$. Since $Y'$ is a $\cc$-$\ca$--bimodule, then the functor $\rho$ is a morphism of dg categories. It induces a morphism between the corresponding opposite dg categories
\[F:\ca\ra\cb'^{op}=:\cb.
\]
Notice that, for each $A'\in\ca$, the functor $(\cc_{dg}\ca)(?,Y_{A'}):\cc\ca\ra\cc k$ induces a triangle functor between the corresponding categories up to homotopy, $(\cc_{dg}\ca)(?,Y_{A'}):\ch\ca\ra\ch k$. Moreover, since $Y_{A'}$ is $\ch$-injective, then this functor induces a triangle functor between the corresponding derived categories
\[(\cc_{dg}\ca)(?,Y_{A'}):\cd\ca\ra\cd k.
\]
When all these functors are applied to the triangles considered above, then we get a family of quasi-isomorphism of complexes of $k$-modules
\[\Psi_{A,A'}:\cc(Y_{A},Y_{A'})\ra Y(A,Y_{A'}),
\]
for each $A, A'\in\ca$. This family underlies a quasi-isomorphism of left dg $\cc$-modules $\Psi_{A,?}:\cc(Y_{A},?)\ra Y(A,?)$ for each $A\in\ca$. Hence, we have a family of quasi-isomorphisms of left dg $\cc$-modules
\[\xi_{A,?}\Psi_{A,?}:\cc(Y_{A},?)\ra Y'(A,?),
\]
for $A\in\ca$. Notice that, since $\ca$ is $k$-flat, for each $A'\in\ca$ the corresponding restriction from $\cc$-$\ca$-bimodules to left dg $\cc$-modules preserves $\ch$-injectives. Indeed, its left adjoint is $?\otimes_{k}\ca$ and preserves acyclic modules. Then for each $A'\in\ca$ the triangle functor
\[\cc_{dg}(\cc^{op})(?,Y'(A',?)):\ch(\cc^{op})\ra\ch k
\]
preserves acyclic modules, and so it induces a triangle functor
\[\cc_{dg}(\cc^{op})(?,Y'(A',?)):\cd(\cc^{op})\ra\cd k.
\]
When applied to the quasi-isomorphisms of left dg $\cc$-modules $\xi_{A,?}\Psi_{A,?}$, for $A\in\ca$, it gives us quasi-isomorphisms of complexes
\[\cb'(Y'(A,?), Y'(A',?))\ra \cc_{dg}(\cc^{op})((Y_{A})^{\we},Y'(A',?)),
\]
and so quasi-isomorphisms of complexes
\[\varepsilon_{}: \cb'(Y'(A,?),Y'(A',?))\ra Y'(A',Y_{A}),
\]
for each $A,A'\in\ca$. It induces a quasi-isomorphism of right dg $\ca$-modules
\[\varepsilon: F_{*}\cb(?,Y'(A,?))\ra Y'(?,Y_{A})
\]
and we get an isomorphism of triangles
\[\xymatrix{X_{A}\ar[r] & A^{\we}\ar[rr]^{\xi_{?,Y_{A}}\varphi_{A}} && Y'(?,Y_{A})\ar[r] & X[1] \\
X_{A}\ar[r]\ar[u]^{\id} & A^{\we}\ar[rr]^{F}\ar[u]^{\id} && F_{*}\cb(?,Y'(A,?))\ar[u]^{\varepsilon}\ar[r] & X[1]\ar[u]^{\id}
}
\]
Notice that from this it follows that $\tau^{\cy}(A^{\we})=F_{*}\cb(?,Y'(A,?))$ for each $A\in\ca$.
By infinite d\'{e}vissage of $\cd\cb$ with respect to the set of modules $B^{\we}\ko  B\in\cb$, the functor $F_{*}:\cd\cb\ra\cd\ca$ has its image in $\cy$. We can assume $X_{A}$ to be the restriction $X(?,A)$ of an $\ca$-$\ca$--bimodule $X$ coming from the triangle induced by $F$ in $\cd(\ca^{op}\otimes_{k}\ca)$. Now, by using the adjunction $(?\otimes^{\bf L}_{\ca}U, F_{*})$ one proves that the functor $?\otimes^{\bf L}_{\ca}U:\cd\ca\ra\cd\cb$ vanishes on the objects of $\cx$, where $U$ is the $\ca$-$\cb$--bimodule defined by $U(B,A):=\cb(B,FA)$. In particular, it vanishes on $X(?,A)$ and by Lemma \ref{Homological epimorphisms of dg categories} we have proved that $F_{*}$ is a homological epimorphism.

Thanks to Lemma \ref{(Super)perfectness and compactness} we know that the modules $F_{*}(B^{\we})\ko B\in\cb$, form a set of compact generators of $\cy$, and so, since $F_{*}$ is fully faithful, the essential image of $F_{*}$ is precisely $\cy$.
\end{proof}

\begin{remark}
The above Theorem is related to subsection 5.4 of K. Br\"{u}ning's Ph. D. thesis \cite{Bruning2007} and section 9 of B. Huber's Ph. D. thesis \cite{Huber2007}. In particular, it generalizes \cite[Corollary 5.4.9]{Bruning2007} and \cite[Corollary 9.8]{Huber2007}. It is also related to B. To\"{e}n's lifting, in the homotopy category of small dg categories up to quasi-equivalences, of \emph{quasi-functors} or \emph{quasi-representable} dg bimodules to genuine dg functors (\cf the proof of \cite[Lemma 4.3]{Toen2005}). \emph{Cf.} also \cite[Lemma 3.2]{Keller1999}.
\end{remark}

Notice that, if $F:\ca\ra\cb$ is a homological epimorphism of dg categories, the bimodule $X$ in the triangle
\[X\ra\ca\arr{F}F_{*}\cb\ra X[1]
\]
of $\cd(\ca^{op}\otimes_{k}\ca)$ has a certain `derived idempotency' expressed by condition 4) of Lemma \ref{Homological epimorphisms of dg categories}, and plays the same r\^ole as the idempotent two-sided ideal in the theory of TTF triples on module categories \cite{Jans, Stenstrom}.

\section{Parametrization for derived categories}\label{Parametrization for derived categories}

Two homological epimorphisms of dg categories, $F:\ca\ra\cb$ and $F':\ca\ra\cb'$, are \emph{equivalent} if the essential images of the corresponding restriction functors, $F_{*}:\cd\cb\ra\cd\ca$ and $F'_{*}:\cd\cb'\ra\cd\ca$, are the same subcategory of $\cd\ca$.

Thanks to Theorem \ref{Homological epimorphisms of dg categories}, we know that every homological epimorphism $F$ starting in a $k$-flat dg category is equivalent to another one $F'$ which is bijective on objects. Unfortunately, the path from $F$ to $F'$ is indirect. Nevertheless, there exists a direct way  of proving (without any flatness assumption) that every homological epimorphism is equivalent to a `quasi-surjective' one.

\begin{proposition}
Every homological epimorphism $F:\ca\ra\cb$ is equivalent to a homological epimorphism $F':\ca\ra\cb'$ such that $\H 0F':\H 0\ca\ra\H 0\cb'$ is essentially surjective.
\end{proposition}
\begin{proof}
Let $\cb'$ be the full subcategory of $\cb$ formed by the objects $B'$ such that $B'^{\we}\cong(FA)^{\we}$ in $\cd\cb$ (equivalently, $B'\cong FA$ in $\H 0\cb$) for some $A\in\ca$. Put $F':\ca\ra\cb'$ for the dg functor induced by $F$. Now, the restriction $j_{*}:\cd\cb\ra\cd\cb'$ along the inclusion $j:\cb'\ra\cb$ is a triangle equivalence. Indeed, thanks to condition 4) of Lemma \ref{Homological epimorphisms of dg categories} we know that the modules $(jB')^{\we}\ko B'\in\cb'$ form a set of compact generators of $\cd\cb$; then one can apply the techniques of \cite[Lemma 4.2]{Keller1994a}. Therefore, the commutative triangle
\[\xymatrix{\cd\cb\ar[dr]^{F_{*}}\ar[d]_{j_{*}}^{\wr} & \\
\cd\cb'\ar[r]_{F'_{*}} & \cd\ca
}
\]
finishes the proof.
\end{proof}

The following lemma shows that Theorem \ref{Homological epimorphisms of dg categories} virtually covers all the possible compactly generated algebraic triangulated categories.

\begin{lemma}
Every compactly generated $k$-linear algebraic triangulated category (whose set of compact generators has cardinality $\lambda$) is triangle equivalent to the derived category of a small $k$-flat dg category (whose set of objects has cardinality $\lambda$).
\end{lemma}
\begin{proof}
Indeed, by B.~Keller's theorem (\cf subsection \ref{B. Keller's Morita theory for derived categories}) we know that it is triangle equivalent to the derived category of a small $k$-linear dg category $\ca$ satisfying the cardinality condition. If $\ca$ is not $k$-flat, then we can consider a cofibrant replacement $F:\ca'\ra\ca$ of $\ca$ in the model structure of the category of small $k$-linear dg categories constructed by G. Tabuada in \cite{Tabuada2005a}, which can be taken to be the identity on objects \cite[Proposition 2.3]{Toen2005}. In particular, $F$ is a quasi-equivalence \cite[subsection 2.3]{Keller2006b}, and so the restriction along $F$ induces a triangle equivalence between the corresponding derived categories \cite[Lemma 3.10]{Keller2006b}. Since $\ca'$ is cofibrant, by \cite[Proposition 2.3]{Toen2005} for all objects $A\ko A'$ in $\ca$ the complex $\ca'(A,A')$ is cofibrant in the category $\cc k$ of complexes over $k$ endowed with its projective model structure \cite[Theorem 2.3.11]{Hovey1999}. This implies that $\ca'(A,A')$ is $\ch$-projective. Finally, since the functor $?\otimes_{k}k$ preserves acyclic complexes and the $\ch$-projective complexes satisfy the principle of infinite d\'{e}vissage with respect to $k$, then $?\otimes_{k}\ca'(A,A')$ preserves acyclic complexes.
\end{proof}

Notice that, in a compactly generated triangulated category, smashing subcategories form a set. This is thanks to H.~Krause's description of them \cite[Corollary 12.5]{Krause2005} and thanks to the fact that isoclasses of compact objects form a set (\cf \cite[Theorem 5.3]{Keller1994a}). Now we will give several descriptions of this set in the algebraic case.

\begin{theorem}
Let $\cd$ be a compactly generated algebraic triangulated category, and let $\ca$ be a $k$-flat dg category whose derived category is triangle equivalent to $\cd$. There exists a bijection between:
\begin{enumerate}[1)]
\item Smashing subcategories $\cx$ of $\cd$.
\item TTF triples $(\cx,\cy,\cz)$ on $\cd$.
\item (Equivalence classes of) recollements for $\cd$.
\item Equivalence classes of homological epimorphisms of dg categories of the form $F:\ca\ra\cb$ (which can be taken to be bijective on objects).
\end{enumerate}
Moreover, if we denote by $\cs$ any of the given (equipotent) sets, then there exists a surjective map $\cR\ra\cs$, where $\cR$ is the class of objects $P\in\cd$ such that $\{P[n]\}^{\bot}_{n\in\Z}$ is closed under small coproducts.
\end{theorem}
\begin{proof}
The bijection between 1) and 2) was recalled in subsection \ref{Smashing subcategories}, and the bijection between 2) and 3) was recalled in subsection \ref{TTF triples and recollements}. The map from 3) to 4) is given by Theorem \ref{Homological epimorphisms of dg categories}, and the map from 4) to 3) is clear from the comments at the beginning of section \ref{Homological epimorphisms of dg categories}. The surjective map $\cR\ra\cs$ follows from Proposition \ref{Parametrization of TTF triples on triangulated categories}.
\end{proof}

\section{Idempotent two-sided ideals}\label{Idempotent two-sided ideals}

\subsection{Generalized smashing conjecture: short survey}\label{Short survey}

The \emph{generalized smashing conjecture} is a generalization to arbitrary compactly generated triangulated categories of a conjecture due to D.~Ravenel \cite[1.33]{Ravenel1984} and, originally, A.~K.~Bousfield \cite[3.4]{Bousfield1979}. It predicts the following:
\vspace*{0.4cm}

\noindent Every smashing subcategory of a compactly generated triangulated category satisfies the principle of infinite d\'{e}vissage with respect to a set of compact objects.
\vspace*{0.5cm}

However, this conjecture was disproved by B.~Keller in \cite{Keller1994}.

In \cite{Krause2000}, H.~Krause gave the definition of being `generated by a class of morphisms':

\begin{definition}
Let $\cd$ be a triangulated category, $\cx$ a strictly full triangulated subcategory of $\cd$ closed under small coproducts and $\ci$ a class of morphisms of $\cd$. We say that $\cx$ is \emph{generated} by $\ci$ if $\cx$ is the smallest full triangulated subcategory of $\cd$ closed under small coproducts and such that every morphism in $\ci$ factors through some object of $\cx$.
\end{definition}

\noindent This is a generalization of the notion of infinite d\'{e}vissage (\cf subsection \ref{Generators and infinite devissage}), at least when the existence of certain countable coproducts is guaranteed, and allows H.~Krause the following reformulation of the generalized smashing conjecture:
\vspace*{0.4cm}

\noindent Every smashing subcategory of a compactly generated triangulated category is ge\-ne\-ra\-ted by a set of identity morphisms between compact objects.
\vspace*{0.5cm}

This reformulation admits at least two weak versions of the conjecture, one of which was proved by H.~Krause \cite[Corollary 4.7]{Krause2000}:
\emph{Every smashing subcategory of a compactly generated triangulated category is generated by a set of morphisms between compact objects.}

In Corollary \ref{General case} below, we prove `the other' weak version of the conjecture, which substitutes ``morphisms between compact objects'' by ``identity morphisms'' of Milnor colimits of compact objects.

\subsection{From ideals to smashing subcategories}\label{From ideals to smashing subcategories}

If $\cd$ is a triangulated category, we denote by $\cd^c$ the full subcategory of $\cd$ formed by the compact objects and by $\cm or(\cd^c)$ the class of morphisms of $\cd^c$. If $\cd$ is the derived category $\cd\ca$ of a dg category $\ca$, then we also write $\cd^c=\cd^c\ca$.

We write $\cp(\cm or(\cd^c))$ for the large complete lattice of
subclasses of $\cm or(\cd^c)$, where the order is given by the
inclusion. Notice that if $\cd$ is compactly generated and we take
the skeleton of $\cd^c$ instead of $\cd^c$ itself we get a set.
Also, we write $\cp(\cd)$ for the large complete lattice of
classes of objects of $\cd$, where the order here is also given by
the inclusion. Consider the following Galois connection (\cf
\cite[section III.8]{Stenstrom} for the definition and basic
properties of Galois connections):
\[\xymatrix{\cp(\cm or(\cd^c))\ar@<1ex>[rr]^{?^{\bot}} && \cp(\cd)\ar@<1ex>[ll]^{\cm or(\cd^c)^{?}},
}
\]
where given $\ci\in\cp(\cm or(\cd^c))$ we define $\ci^{\bot}$ to be the class of objects $Y$ of $\cd$ such that $\cd(f,Y)=0$ for every morphism $f$ of $\ci$, and given $\cy\in\cp(\cd)$ we define $\cm or(\cd^c)^{\cy}$ to be the class of morphisms $f$ of $\cd^c$ such that $\cd(f,Y)=0$ for every $Y$ in $\cy$.

According to the usual terminology in the theory of Galois
connections, we say that a subclass $\ci$ of $\cm or(\cd^c)$ is
\emph{closed} if $\ci=\cm or(\cd^c)^{\ci^{\bot}}$, and a class
$\cy$ of objects of $\cd$ is \emph{closed} if $\cy=(\cm
or(\cd^c)^{\cy})^{\bot}$. Notice that a subclass $\ci$ of $\cm
or(\cd^c)$ is closed if and only if $\ci=\cm or(\cd^c)^{\cy}$ for
some class $\cy$ of objects of $\cd$. Similarly, a class $\cy$ of
objects of $\cd$ is closed if and only if $\cy=\ci^{\bot}$ for
some class $\ci$ of morphisms of $\cd^c$.

If $\ci$ is a subclass of $\cm or(\cd^c)$, we write $\ci[1]$ for
the class formed by all the morphisms of the form $f[1]$ with $f$
in $\ci$.

The following lemma is an easy exercise.

\begin{lemma}
The Galois connection above induces a bijection between the class formed by the closed ideals $\ci$ of $\cd^c$ such that $\ci[1]=\ci$ and the class formed by the full triangulated subcategories of $\cd$ closed under small coproducts which are closed for the Galois connection.
\end{lemma}

In the following result we explain how a certain two-sided ideal induces a nice smashing subcategory or, equivalently, a TTF triple.

\begin{theorem}\label{from ideals to devissage wrt hc}
Let $\cd$ be a compactly generated triangulated category and let $\ci$ be an idempotent two-sided ideal of $\cd^c$ with $\ci[1]=\ci$. There exists a triangulated TTF triple $(\cx,\cy,\cz)$ on $\cd$ such that:
\begin{enumerate}[1)]
\item $\cx=\Tria(\cp)$, for a certain set $\cp$ of Milnor colimits of sequences of morphisms of $\ci$.
\item $\cy=\ci^{\bot}$.
\item A morphism of $\cd^c$ belongs to $\ci$ if and only if it factors through an object of $\cp$.
\end{enumerate}
\end{theorem}
\begin{proof}
Theorem 5.3 of \cite{Keller1994a} implies that $\cd^c$ is skeletally small. Fix a small skeleton of $\cd^c$ closed under shifts. Let $\cp$ be the set of Milnor colimits of all the sequences of morphisms of $\ci$ between objects of the fixed skeleton. Remark that $\cp$ is closed under shifts. Put $(\cx,\cy,\cz):=(\Tria(\cp),\Tria(\cp)^{\bot},(\Tria(\cp)^{\bot})^{\bot})$.

\emph{First step: $\Tria(\cp)^{\bot}=\ci^{\bot}$.} Notice that $\Tria(\cp)^{\bot}$ is the class of those objects which are right orthogonal to all the objects of $\cp$. Hence, for the inclusion $\ci^{\bot}\subseteq\Tria(\cp)^{\bot}$, it suffices to prove $\cd(P,Y)=0$ for each $P\in\cp$ and $Y\in\ci^{\bot}$. For this, let
\[P_{0}\arr{f_{0}}P_{1}\arr{f_{1}}P_{2}\ra\dots
\]
be a sequence of morphisms in $\ci$ between objects of the fixed skeleton and consider the corresponding Milnor triangle
\[\coprod_{i\geq 0}P_{i}\arr{\id-\sigma}\coprod_{i\geq 0}P_{i}\ra P\ra\coprod_{i\geq 0}P_{i}[1].
\]
Since $\cd(\sigma,Y[n])=0$ for each $n\in\Z$, we get a long exact sequence
\[\dots \arr{\id}\cd(\coprod_{i\geq 0}P_{i}[1],Y)\ra\cd(P,Y)\ra\cd(\coprod_{i\geq 0}P_{i},Y)\arr{\id}\dots
\]
which proves that $\cd(P,Y)=0$. Conversely, let $Y\in\Tria(\cp)^{\bot}$ and consider a morphism $f:P\ra P'$ of $\ci$. Put $P_{0}:=P$ and, by using the idempotency of $\ci$, consider a factorization of $f$
\[\xymatrix{P_{0}\ar[rr]^f\ar[dr]_{f_{0}} && P' \\
& P_{1}\ar[ur]_{g_{1}}
}
\]
with $f_{0}\ko g_{1}\in\ci$. We can consider a similar factorization for $g_{1}$, and proceeding inductively we can produce a sequence of morphisms of $\ci$
\[P_{0}\arr{f_{0}} P_{1}\arr{f_{1}} P_{2}\ra\dots
\]
together with morphisms $g_{i}:P_{i}\ra P'$ of $\ci$ satisfying $g_{i}f_{i-1}=g_{i-1}\ko i\geq 1$ with $g_{0}:=f$.
\[\xymatrix{P=P_{0}\ar[r]^{f_{0}}\ar[dr]_{f=g_{0}} & P_{1}\ar[r]^{f_{1}}\ar[d]^{g_{1}} & P_{2}\ar[r]\ar[dl]^{g_{2}} & \dots \\
& P' &&
}
\]
This induces a factorization of $f$
\[f:P\ra \Mcolim P_{n}\ra P'
\]
Since $\Mcolim P_{n}$ is isomorphic to an object of $\cp$, then $(\cd\ca)(f,Y)=0$. This proves that $\Tria(\cp)^{\bot}=\ci^{\bot}$.

\emph{Second step: $(\cx,\cy,\cz)$ is a triangulated TTF triple.} Thanks to Corollary \ref{Smashing subcategories} it suffices to prove that $\cx$ is a smashing subcategory of $\cd$, \ie that $\cx$ is an aisle in $\cd$ and $\cy$ is closed under small coproducts. The fact that $\cx$ is an aisle follows from \cite[Corollary 3.12]{Porta2007}. The fact that $\cy$ is closed under small coproducts follows from the first step, since the morphisms of $\ci$ are morphisms between \emph{compact} objects.

\emph{Third step: part 3)}. Notice that in the proof of the inclusion $\Tria(\cp)^{\bot}\subseteq\ci^{\bot}$ we have showed that every morphism of $\ci$ factors through an object of $\cp$. The converse is also true. Indeed, let $g:Q'\ra Q$ be a morphism between compact objects factoring through an object $P$ of $\cp$:
\[\xymatrix{& Q'\ar[d]^{g}\ar@{.>}[dl]_{h} \\
P\ar[r] & Q
}
\]
By definition of $\cp$, we have that $P$ is the Milnor colimit of a sequence of morphisms of $\ci$:
\[P=\Mcolim(P_{0}\arr{f_{0}}P_{1}\arr{f_{1}}P_{2}\ra\dots)
\]
Now compactness of $Q'$ implies that $h$ factors through a certain $P_{n}$:
\[\xymatrix{&& Q'\ar[d]^{g}\ar@{.>}[dll]_{h_{n}} \\
P_{n}\ar[r]_{\pi_{n}}& P\ar[r] & Q
}
\]
where $\pi_{n}$ is the $n$th component of the morphism $\pi$ appearing in the Milnor triangle defining $P$. One of the properties satisfied by the components of $\pi$ is the identity: $\pi_{m+1}f_{m}=\pi_{m}$ for each $m\geq 0$. This gives the following factorization for $g$:
\[\xymatrix{&&& Q'\ar[d]^{g}\ar@{.>}[dlll]_{h_{n}} \\
P_{n}\ar[r]_{f_{n}}&P_{n+1}\ar[r]_{\pi_{n+1}}& P\ar[r] & Q
}
\]
Since $f_{n}$ belongs to the ideal $\ci$, so does $g$.
\end{proof}

\subsection{Stretching a filtration}\label{Stretching a filtration}

The results of this subsection are to be used in the proof of Proposition \ref{From smashing subcategories to ideals}. Throughout this subsection $\ca$ will be a small dg category.

We define $\cs$ to be the set of right dg $\ca$-modules $S$ admitting a finite filtration
\[0=S_{-1}\subset S_{0}\subset S_{1}\subset\dots S_{n}=S
\]
in $\cc\ca$ such that
\begin{enumerate}[1)]
\item the inclusion morphism $S_{p-1}\subset S_{p}$ is an inflation for each $0\leq p\leq n$,
\item the factor $S_{p}/S_{p-1}$ is isomorphic in $\cc\ca$ to a \emph{relatively free module of finite type} (\ie it is a finite coproduct of modules of the form $A^{\we}[i]\ko A\in\ca\ko i\in\Z$) for each $0\leq p\leq n$.
\end{enumerate}

\begin{lemma}
\begin{enumerate}[1)]
\item The compact objects of $\cd\ca$ are precisely the direct summands of objects of $\cs$.
\item For each $S\in\cs$ the functor $(\cd\ca)(S,q?):\cc\ca\ra\Mod k$ preserves direct limits, where $q:\cc\ca\ra\cd\ca$ is the canonical localization functor.
\item For each $S\in\cs$ the functor $(\cc\ca)(S,?):\cc\ca\ra\Mod k$ preserves direct limits.
\end{enumerate}
\end{lemma}
\begin{proof}
1) We know by \cite[Theorem 5.3]{Keller1994a} that any compact object $P$ of $\cd\ca$ is a direct summand of a finite extension of objects of the form $A^{\we}[n]\ko A\in\ca\ko n\in\Z$. By using that every triangle of $\cd\ca$ is isomorphic to a triangle coming from a conflation of $\cc\ca$, we have that this kind of finite extensions are objects isomorphic in $\cd\ca$ to objects of $\cs$.

2) Let $I$ be a directed set. We regard it as a category and denote by $\Fun(I,?)$ the category of functors starting in $I$. We want to prove that, for every object $S\in\cs$, the following square is commutative:
\[\xymatrix{\Fun(I,\cc\ca)\ar[r]^{\lid}\ar[d]_{(I,(\cd\ca)(S,q?))} & \cc\ca\ar[d]^{(\cd\ca)(S,q?)}\\
\Fun(I,\Mod k)\ar[r]^{\lid}&\Mod k
}
\]
\emph{First step:}
We first prove it for $S=A^{\we}\ko A\in\ca$. That will prove assertion 2) for relatively free dg $\ca$-modules of finite type.
For an object $A$ of $\ca$ we consider the dg functor $F:k\ra\ca$, with $F(k)=A$ and $F(1_{k})=\id_{A}$. The restriction
\[F_{*}:\cc\ca\ra\cc k\ko M\mapsto M(A)
\]
along $F$ admits a right adjoint, and so it preserves colimits. Consider the commutative diagram
\[\xymatrix{\Fun(I,\cc\ca)\ar[rr]^{\lid}\ar[d]^{(I,F_{*})}\ar@/_3pc/[dd]_{(I,(\cd\ca)(A^{\we},q?))} && \cc\ca\ar[d]_{F_{*}}\ar@/^3pc/[dd]^{(\cd\ca)(A^{\we},q?)} \\
\Fun(I,\cc k)\ar[rr]^{\lid}\ar[d]^{(I,\H 0)} && \cc k\ar[d]_{\H 0}  \\
\Fun(I,\Mod k)\ar[rr]_{\lid}&&\Mod k
}
\]
Since $F_{*}$ preserves colimits, the top rectangle commutes. Since
\[\lid:\Fun(I,\Mod k)\ra\Mod k
\]
is exact, the bottom rectangle commutes. Indeed,
\begin{align}
\lid\H 0M_{i}=\lid\cok(\Bo 0M_{i}\ra \Zy 0M_{i})=\cok(\lid(\Bo 0M_{i}\ra \Zy 0M_{i}))= \nonumber \\
=\cok(\lid\Bo 0M_{i}\ra\lid\Zy 0M_{i})=\cok(\Bo 0(\lid M_{i})\ra\Zy 0(\lid M_{i}))=\H 0\lid M_{i}. \nonumber
\end{align}

\emph{Second step:} Let $S'\ra S\ra S''$ be a conflation of $\cc\ca$ such that $S'$ and $S''$ (and hence all their shifts) satisfy property 2), and let $M\in\Fun(I,\cc\ca)$ be a direct system.
By using the cohomological functors $(\cd\ca)(?,M_{i})\ko (\cd\ca)(?,\lid M_{i})$, the fact that $\lid:\Fun(I,\Mod k)\ra\Mod k$ is exact and the five lemma, we get an isomorphism
\[\lid(\cd\ca)(S, M_{i})\arr{\sim} (\cd\ca)(S, \lid M_{i})
\]

3) By using the same techniques as in 2), one proves that relatively free dg modules of finite type have the required property. Now, let $S'\arr{j}S\arr{p}S''$ be a conflation such that $S'$ satisfies the required property and $S''$ is relatively free of finite type. For each $N\in\cc\ca$, this conflation gives an exact sequence
\[0\ra(\cc\ca)(S'',N)\arr{p^{\che}}(\cc\ca)(S,N)\arr{j^{\che}}(\cc\ca)(S',N).
\]
Let $u\in(\cc\ca)(S''[-1],S')$ be a morphism whose mapping cone is $\cone(u)=S$. By using the triangle
\[S''[-1]\arr{u}S'\arr{j}S\arr{p}S''
\]
and the canonical localization $q:\cc\ca\ra\cd\ca$ we can fit the sequence above in the following commutative diagram
\[\xymatrix{0\ar[d] & \\
(\cc\ca)(S'',N)\ar[d]^{p^{\che}} & \\
(\cc\ca)(S,N)\ar[d]^{j^{\che}}\ar@{->>}[r]^{q} & (\cd\ca)(S,N)\ar[d]^{j^{\che}} \\
(\cc\ca)(S',N)\ar[d]^{\phi}\ar@{->>}[r]^{q} & (\cd\ca)(S',N)\ar[d]^{u^{\che}} \\
(\cd\ca)(S''[-1],N)\ar[r]^{\id} & (\cd\ca)(S''[-1],N)
}
\]
The aim is to prove that the left-hand column is exact. Indeed, $\varphi j^{\che}=0$ since $u^{\che}j^{\che}=0$.
It only remains to prove that $\ker\varphi$ is contained in $\im(j^{\che})$. For this, let $f\in(\cc\ca)(S',N)$ be such that $\varphi(f)=fu$ vanishes in
\[(\cd\ca)(S''[-1],N)\underset{\sim}{\leftarrow}(\ch\ca)(S''[-1],N),
\]
\ie $fu$ is null-homotopic. By considering the triangle above, one sees that $f$ factor through $j$ in $\ch\ca$, \ie there exists $g\in(\cc\ca)(S,N)$ such that $f-gj$ is null-homotopic. Then, there exists $h\in(\cc\ca)(IS',N)$ such that $f=gj+hi_{S'}$, where $i_{S'}:S'\ra IS'$ is an inflation to an injective module (for the graded-wise split exact structure of $\cc\ca$). Since $j$ is an inflation, then there exists $g'$ such that $g'j=i_{S'}$. Therefore, $f=gj+hi_{S'}=(g+hg')j$ belongs to $\im(j^{\che})$.

Finally, let $M\in\op{Fun}(I,\cc\ca)$ be a direct system. Using assertion 2) and the hypothesis on $S'$ and $S''$, we have a morphism of exact sequences:
\[\xymatrix{0\ar[d] & 0\ar[d] \\
\lid(\cc\ca)(S'',M_{i})\ar[r]^{\sim}\ar[d] & (\cc\ca)(S'',\lid M_{i})\ar[d] \\
\lid(\cc\ca)(S,M_{i})\ar[r]\ar[d] & (\cc\ca)(S,\lid M_{i})\ar[d] \\
\lid(\cc\ca)(S',M_{i})\ar[r]^{\sim}\ar[d] & (\cc\ca)(S',\lid M_{i})\ar[d] \\
\lid(\cd\ca)(S''[-1],M_{i})\ar[r]^{\sim} & (\cd\ca)(S''[-1],\lid M_{i})
}
\]
where the horizontal arrows are the natural ones. Hence, the second horizontal arrow is an isomorphism.
\end{proof}

\begin{proposition}
Every $\ch$-projective right dg $\ca$-module is, up to isomorphism in $\ch\ca$, the colimit in $\cc\ca$ of a direct system of submodules $S_{i}\ko i\in I$, such that:
\begin{enumerate}[1)]
\item $S_{i}\in\cs$ for each $i\in I$,
\item for each $i\leq j$ the morphism $\mu^{i}_{j}:S_{i}\ra S_{j}$ is an inflation.
\end{enumerate}
\end{proposition}
\begin{proof}
\emph{First step:} Assume that an object $P$ is the colimit of a direct system $P_{t}\ko t\in T$, of subobjects such that the structure morphisms $\mu^r_{t}: P_{r}\ra P_{t}$ are inflations and each $P_{t}$ is the colimit of a direct system $S_{(t,i)}\ko i\in I_{t}$, of subobjects satisfying the following property $(*)$:
\begin{enumerate}[1)]
\item $S_{(t,i)}\in\cs$ for each $i\in I_{t}$,
\item for each $i$ the morphism $\mu_{(t,i)}: S_{(t,i)}\ra P_{t}$ is an inflation.
\end{enumerate}
Notice that 2) implies that for each $i\leq j$ the structure morphism $\mu^{i}_{j}:S_{(t,i)}\ra S_{(t,j)}$ is an inflation.

Put $I:=\bigcup_{t\in T}(\{t\}\times I_{t})$ and define the following preorder: $(r,i)\leq (t,j)$ if $r\leq t$ and we have a factorization as follows
\[\xymatrix{S_{(r,i)}\ar@{.>}[rr]^{\mu^{(r,i)}_{(t,j)}}\ar[d]_{\mu_{(r,i)}} && S_{(t,j)}\ar[d]^{\mu_{(t,j)}} \\
P_{r}\ar[rr]_{\mu^{r}_{t}} && P_{t}
}
\]
Notice that $\mu^{(r,i)}_{(t,j)}$ is an inflation. To prove that $I$ is a directed preordered set take $(r,i)\ko (t,j)$ in $I$ and let $s\in T$ with $r\ko t\leq s$. Since $S_{(r,i)}$ and $S_{(t,j)}$ are in $\cs$, thanks to Lemma \ref{Stretching a filtration} we know that the compositions $S_{(r,i)}\ra P_{r}\ra P_{s}$ and $S_{(t,j)}\ra P_{t}\ra P_{s}$ factors through $S_{(s,i_{r})}\ra P_{s}$ and $S_{(s,j_{t})}\ra P_{s}$. Now, if $i_{r}\ko j_{t}\leq k$ we have that $(r,i)\ko (t,j)\leq (s,k)$. Now, take the quotient set $I':=I/\sim$, where $(r,i)\sim (r',i')$ when $(r,i)\leq (r',i')\leq (r,i)$. Notice that in this case $r=r'$ and the inflations $\mu^{(r,i)}_{(r,i')}$ and $\mu^{(r,i')}_{(r,i)}$ are mutually inverse. Thus, $S_{(t,i)}\ko [(t,i)]\in I'$ is a direct system of subobjects of $P$ which are in $\cs$ and whose colimit is easily seen to be $P$.

\emph{Second step:} Let $P$ be an $\ch$-projective module. The proof of \cite[Theorem 3.1]{Keller1994a} shows that, up to replacing $P$ by a module isomorphic to it in $\ch\ca$, we can consider a filtration
\[0=P_{-1}\subset P_{0}\subset P_{1}\subset\dots P_{n}\subset P_{n+1}\dots\subset P\ko n\in\N
\]
such that
\begin{enumerate}[1)]
\item $P$ is the union of the $P_{n}\ko n\in\N$,
\item the inclusion morphism $P_{n-1}\subset P_{n}$ is an inflation for each $n\in\N$,
\item the factor $P_{n}/P_{n-1}$ is isomorphic in $\cc\ca$ to a \emph{relatively free module} (\ie, a direct sum of modules of the form $A^{\we}[i]\ko A\in\ca\ko i\in\Z$) for each $n\in\N$.
\end{enumerate}
Thanks to the first step, it suffices to prove that each $P_{n}$ is the colimit of a direct system of subobjects satisfying $(*)$. We will prove it inductively.

\emph{Third step:} Notice that $P_{n}=\cone(f)$ for a morphism $f:L\ra P_{n-1}$, where $L=\bigoplus_{I}A^{\we}_{i}[n_{i}]\ko A_{i}\in\ca\ko n_{i}\in\Z$. By hypothesis of induction, $P_{n-1}=\lid S_{j}$ where $S_{j}\ko j\in J$ is a direct system satisfying $(*)$. Let $\cF\cp(I)$ be the set of finite subsets of $I$, and put $L_{F}=\bigoplus_{i\in F}A^{\we}_{i}[n_{i}]$ for $F\in\cF\cp(I)$. Notice that $\cF\cp(I)$ is a directed set with the inclusion, and that $L=\lid L_{F}$. Consider the set $\Omega$ of pairs $(F,j)\in\cF\cp(I)\times J$ such that there exists a morphism $f_{(F,j)}$ making the following diagram commutative:
\[\xymatrix{L_{F}\ar[r]^{f_{(F,j)}}\ar[d]_{u_{F}} & S_{j}\ar[d]^{\mu_{j}} \\
L\ar[r]_{f} & P_{n-1}
}
\]
$\Omega$ is a directed set with the order: $(F,j)\leq (F',j')$ if and only if $F\subseteq F'$ and $j\leq j'$. Let $\mu^j_{j'}:S_{j}\ra S_{j'}$ and $u^F_{F'}:L_{F}\ra L_{F'}$ be the structure morphisms of the direct systems $S_{j}\ko j\in J$ and $L_{F}\ko F\in\cF\cp(I)$. Then one can check that $\cone(f_{(F,j)})\ko (F,j)\in\Omega$ is a direct system of modules, with structure morphisms
\[\left[\begin{array}{cc}\mu^j_{j'}& 0 \\ 0 & u^F_{F'}[1]\end{array}\right]:\cone(f_{(F,j)})\ra\cone(f_{(F',j')}),
\]
and whose colimit is $\cone(f)=P_{n}$ via the morphisms
\[\left[\begin{array}{cc}\mu_{j}& 0 \\ 0 & u_{F}[1]\end{array}\right]:\cone(f_{(F,j)})\ra\cone(f)=P_{n},
\]
which are inflations. Finally, notice that, since there exists a conflation
\[S_{j}\ra\cone(f_{(F,j)})\ra L_{F}[1]
\]
with $S_{j}\in\cs$ and $L_{F}[1]$ relatively free of finite type, then each $\cone(f_{(F,j)})$ is in $\cs$.
\end{proof}

\subsection{From smashing subcategories to ideals}\label{From smashing subcategories to ideals}

\subsubsection{H. Krause's bijection}\label{H. Krause's bijection}

Recall (\cf \cite[subsection 2.3]{Beligiannis2000} or \cite[Definition 8.3]{Krause2005}) that a two-sided ideal $\ci$ of a triangulated category $\cd$ is \emph{saturated} if whenever there exists a triangle
\[P'\arr{u} P\arr{v} P''\ra P'[1]
\]
in $\cd$ and a morphism $f\in\cd(P,Q)$ with $fu\ko v\in\ci$, then $f\in\ci$.

After Theorem 11.1, Theorem 12.1, Corollary 12.5 and Corollary 12.6 of the article \cite{Krause2005} of H.~Krause, an taking into account subsection \ref{Smashing subcategories}, one has the following nice bijection:

\begin{theorem}
If $\cd$ is a compactly generated triangulated category, the maps
\[(\cx,\cy,\cz)\mapsto\cm or(\cd^c)^{\cy}
\]
and
\[\ci\mapsto(^{\bot}(\ci^{\bot}),\ci^{\bot}, (\ci^{\bot})^{\bot})
\]
define a bijection between the set of TTF triples on $\cd$ and the set of saturated idempotent two-sided ideals $\ci$ of $\cd^c$ such that $\ci[1]=\ci$.
\end{theorem}

\subsubsection{General case}\label{General case}

Let us deduce a bijection in the spirit of the above theorem
directly from Theorem \ref{From ideals to smashing subcategories}
and the following result of H.~Krause:

\begin{proposition}
Let $\cx$ be a smashing subcategory of a compactly generated triangulated category $\cd$. If $P$ is a compact object of $\cd$, $M$ is an object of $\cx$ and $f:P\ra M$ is a morphism of $\cd$, then there exists a factorization in $\cd$
\[\xymatrix{P\ar[rr]^{f}\ar[dr]_{u} && M \\
& P'\ar[ur] &
}
\]
with $P'$ compact and $u$ factoring through an object of $\cx$.
\end{proposition}
\begin{proof}
This statement is assertion 2') of \cite[Theorem 4.2]{Krause2000}. We will give here the proof for the algebraic setting by using that, in this case, every smashing sub\-ca\-te\-go\-ry is induced by a homological epimorphism. First, recall that the smashing subcategory $\cx$ fits into a TTF triple $(\cx,\cy,\cz)$ on $\cd$. Thanks to Lemma \ref{Parametrization for derived categories} we know that there exists a small $k$-flat dg category $\ca$ whose derived category $\cd\ca$ is triangle equivalent to $\cd$. Also, thanks to Lemma \ref{Stretching a filtration} we can assume that $P$ belongs to the set $\cs$ defined in subsection \ref{Stretching a filtration}. Let $F:\ca\ra\cb$ be a homological epimorphism of dg categories associated (\cf Theorem \ref{Homological epimorphisms of dg categories}) to the TTF triple $(\cx,\cy,\cz)$, and fix a triangle
\[X\arr{\alpha}\ca\arr{F}F_{*}\cb\ra X[1]
\]
in $\cd(\ca^{op}\otimes_{k}\ca)$. Assume $M$ is $\ch$-projective, and let $S_{i}\ko i\in I$, be a direct system of submodules of $M$ as in Proposition \ref{Stretching a filtration} so that $M=\lid S_{i}$. Then, we get a direct system $S_{i}\otimes_{\ca}X\ko i\in I$, such that $\lid (S_{i}\otimes_{\ca}X)\cong M\otimes_{\ca}X$. Since $M\in\cx$, for each $i\in I$ we have a commutative square
\[\xymatrix{(\cd\ca)(P,S_{i}\otimes_{\ca}X)\ar[r]\ar[d] &(\cd\ca)(P,M\otimes_{\ca}X)\ar[d]^{\wr} \\
(\cd\ca)(P,S_{i})\ar[r] &(\cd\ca)(P,M)
}
\]
where the horizontal arrows are induced by the morphisms $\mu_{i}:S_{i}\ra M$ associated to the colimit and the vertical arrows are induced by the compositions
\[S_{i}\otimes_{\ca}X\overset{\id\otimes\alpha}{\longrightarrow}S_{i}\otimes_{\ca}\ca\underset{\sim}{\longrightarrow}S_{i}
\]
and
\[M\otimes_{\ca}X\underset{\sim}{\overset{\id\otimes\alpha}{\longrightarrow}}M\otimes_{\ca}\ca\underset{\sim}{\longrightarrow}M.
\]
According to Lemma \ref{Stretching a filtration}, there exists an index $i\in I$ such that $f$ comes from a morphism $(\cd\ca)(P,S_{i}\otimes_{\ca}X)$ via the square above. Then, $f$ factors in $\cd\ca$ as
\[f:P\ra S_{i}\otimes_{\ca}X\ra S_{i}\arr{\mu_{i}}M
\]
and the result follows from the fact that $S_{i}\in\cs$ and $S_{i}\otimes_{\ca}X\in\cx$ (see the comments immediately after Lemma \ref{Homological epimorphisms of dg categories}).
\end{proof}

Here we have the promised bijection:

\begin{theorem}\label{our result}
If $\cd$ is a compactly generated triangulated category, the maps
\[(\cx,\cy,\cz)\mapsto\cm or(\cd^c)^{\cy}
\]
and
\[\ci\mapsto(^{\bot}(\ci^{\bot}),\ci^{\bot}, (\ci^{\bot})^{\bot})
\]
define a bijection between the set of all the TTF triples on $\cd$ and the set of all closed idempotent two-sided ideals $\ci$ of $\cd^c$ such that $\ci[1]=\ci$. \end{theorem}
\begin{proof}
Let $(\cx,\cy,\cz)$ be a TTF triple on the category $\cd$ and put $\ci:=\cm or(\cd^c)^{\cy}$. Bearing in mind the Galois connection of subsection \ref{From ideals to smashing subcategories} it only remains to prove that $\ci$ is idempotent and $\ci^{\bot}=\cy$.

It is easy to check that $\ci$ is precisely the class of all morphisms of $\cd^c$ which factor through an object of $\cx$. Then, by using the above proposition we prove that $\ci$ is idempotent. Finally, let us check that $\ci^{\bot}=\cy$. Of course, it is clear that  $\cy$ is contained in $\ci^{\bot}$. Conversely, let $M$ be an object of $\ci^{\bot}$ and consider the triangle
\[x\tau_{\cx}M\ra M\ra y\tau^{\cy}M\ra x\tau_{\cx}M[1].
\]
Since both $M$ and $y\tau^{\cy}M$ belongs to $\ci^{\bot}$, then $x\tau_{\cx}M$ belongs to $\cx\cap\ci^{\bot}$. Since the compact objects generate $\cd$, if $x\tau_{\cx}M\neq 0$ there exists a non-zero morphism
\[f:P\ra x\tau_{\cx}M
\]
for some compact object $P$. Thanks to the above proposition, we know that $f$ admits a factorization $f=vu$ through a compact object with $u$ in $\ci$, and so $f=0$. This contradiction implies $x\tau_{\cx}M=0$ and thus $M\in\cy$.
\end{proof}

\begin{remark}
When $\cd$ is the derived category of a small dg category, the idempotency of the ideals of the above theorem reflects the `derived idempotency' of the bimodule appearing in the characterization of homological epimorphisms (\cf Lemma \ref{Homological epimorphisms of dg categories}).
\end{remark}

As a consequence of Theorems \ref{From ideals to smashing subcategories} and \ref{General case} we get our announced weak version of the generalized smashing conjecture:

\begin{corollary}
Every smashing subcategory of a compactly generated triangulated category satisfies the principle of infinite d\'evissage with respect to a set of Milnor colimits of compact objects.
\end{corollary}

\subsubsection{Algebraic case}\label{Algebraic case}

For compactly generated algebraic triangulated categories we get an alternative proof of H.~Krause's bijection as stated in Theorem \ref{H. Krause's bijection}.

\begin{corollary}
Let $\cd$ be a compactly generated algebraic triangulated category. Then
\[(\cx,\cy,\cz)\mapsto\cm or(\cd^c\ca)^{\cy}
\]
and
\[\ci\mapsto(^{\bot}(\ci^{\bot}),\ci^{\bot}, (\ci^{\bot})^{\bot})
\]
define a bijection between the set of all the TTF triples on $\cd$ and the set of all saturated idempotent two-sided ideals $\ci$ of $\cd^c$ such that $\ci[1]=\ci$.
\end{corollary}
\begin{proof}
It is easy to check that every closed ideal of $\cd^c$ is saturated. Therefore, thanks to Theorem \ref{General case} we just have to prove that, in this case, every saturated idempotent two-sided ideal $\ci$ of $\cd^c$ such that $\ci=\ci[1]$ is closed. For this, let $\ci$ be a saturated idempotent two-sided ideal of $\cd^c$ with $\ci[1]=\ci$. Consider the TTF triple $(\cx,\cy,\cz):=(^{\bot}(\ci^{\bot}),\ci^{\bot},(\ci^{\bot})^{\bot})$ associated to $\ci$ in Theorem \ref{from ideals to devissage wrt hc}, and let $\cp$ be a `set' of Milnor colimits of sequences of morphisms of $\ci$ as in the proof of that theorem. In particular, $\cp$ is closed under shifts and $\Tria(\cp)=\cx$. Put $\ci':=\cm or(\cd^c)^{\cy}$. Of course, $\ci\subseteq\ci'$. The aim is to prove the converse inclusion, which would imply that $\ci$ is closed. Let $f:Q'\ra Q$ be a morphism of $\ci'$ and consider the triangle
\[x\tau_{\cx}Q\arr{\delta}Q\arr{\eta}y\tau^{\cy}Q\ra x\tau_{\cx}Q[1].
\]
Since $\eta f=0$, then $f$ factors through $\delta$ via the following dotted arrow:
\[\xymatrix{ & Q'\ar[d]^{f}\ar@{.>}[dl]\ar[dr]^{0} & & \\
x\tau_{\cx}Q\ar[r]_{\delta} & Q\ar[r]_{\eta} & y\tau^{\cy}Q\ar[r] & x\tau_{\cx}Q[1]
}
\]
Theorem 4.3 of \cite{Keller1994a} says that we can assume that $\cd$ is the derived category $\cd\ca$ of a small dg category $\ca$. Also, Lemma \ref{Stretching a filtration} allows us to assume that $Q'$ belongs to the set $\cs$ described in subsection \ref{Stretching a filtration}. Thanks to Theorem \ref{Recollement-defining sets in aisled categories} we know that $x\tau_{\cx}Q$ can be taken to be, in the category of dg $\ca$-modules, the direct limit of a certain $\lambda$-sequence $X:\lambda\ra\cc\ca$, \ie a colimit-preserving functor $X$ starting in an ordinal $\lambda$ (\cf \cite[Definition 2.1.1]{Hovey1999}), such that:
\begin{enumerate}[-]
\item $X_{0}=0$,
\item for all $\alpha<\lambda$, the morphism $X_{\alpha}\ra X_{\alpha+1}$ is an inflation with cokernel in $\cp$.
\end{enumerate}
Then Lemma \ref{Stretching a filtration} implies that $f$ factors through a certain $X_{\alpha}$:
\[\xymatrix{ && Q'\ar[d]^{f}\ar@{.>}[dll]\ar[dr]^{0} & & \\
X_{\alpha}\ar[r]_{\can\ \ }&x\tau_{\cx}Q\ar[r]_{\delta} & Q\ar[r]_{\eta} & y\tau^{\cy}Q\ar[r] & x\tau_{\cx}Q[1]
}
\]
Let us prove, by transfinite induction on $\alpha$, that if a morphism of $\ci'$ factors through some $X_{\alpha}$ then it belongs to $\ci$.

\emph{First step:} For $\alpha=0$ it is clear since $X_{0}=0$.

\emph{Second step: Assume that every morphism of $\ci'$ factoring through $X_{\alpha}$ belongs to $\ci$.} Let $f$ be a morphism of $\ci'$ factoring through $X_{\alpha+1}$:
\[\xymatrix{Q'\ar[rr]^{f}\ar[dr]_{u} && Q \\
&X_{\alpha+1}\ar[ur]_{v}&
}
\]
In $\cd$ we have a triangle
\[X_{\alpha}\arr{a}X_{\alpha+1}\arr{b} P\arr{\gamma} X_{\alpha}[1]
\]
with $P\in\cp$, \ie $P$ is the Milnor colimit of a sequence
\[P_{0}\arr{g_{0}}P_{1}\arr{g_{1}}P_{2}\ra\dots
\]
of morphisms of $\ci$. By using compactness of $Q'$ we get a commutative diagram
\[\xymatrix{& Q'\ar[r]^{f}\ar[d]_{u}\ar@/_5pc/[dd]_{w} & Q & \\
X_{\alpha}\ar[r]^{a} & X_{\alpha+1}\ar[r]^{b}\ar[ur]^{v} & P\ar[r]^{\gamma} & X_{\alpha}[1] \\
& P_{t}\ar[ur]^{\pi_{t}}\ar[r]_{g_{t}} & P_{t+1}\ar[u]_{\pi_{t+1}}&
}
\]
where $\pi_{n}$ is the $n$th component of the morphism $\pi$ appearing in the definition of Milnor colimit. Consider the following commutative diagram in which the rows are triangles
\[\xymatrix{X_{\alpha}\ar[r]^{a_{t}}\ar@{=}[d] & M_{t}\ar[r]^{b_{t}}\ar[d]^{\psi} & P_{t}\ar[r]^{\gamma\pi_{t}}\ar[d]^{g_{t}} & X_{\alpha}[1]\ar@{=}[d] \\
X_{\alpha}\ar[r]^{a_{t+1}}\ar@{=}[d] & M_{t+1}\ar[r]^{b_{t+1}}\ar[d]^{\phi} & P_{t+1}\ar[r]^{\gamma\pi_{t+1}}\ar[d]^{\pi_{t+1}} & X_{\alpha}[1]\ar@{=}[d] \\
X_{\alpha}\ar[r]^{a} & X_{\alpha+1}\ar[r]^{b} & P\ar[r]^{\gamma} & X_{\alpha}[1]
}
\]
Since $\gamma\pi_{t}w=\gamma bu=0$, then $w=b_{t}w'$ for some morphism $w':Q'\ra M_{t}$. Therefore, $bu=\pi_{t}w=\pi_{t}b_{t}w'=b\phi\psi w'$, that is to say, $b(u-\phi\psi w')=0$ and so $u-\phi\psi w'=a\xi=\phi\psi a_{t}\xi$ for a certain morphism $\xi:Q'\ra X_{\alpha}$. Hence, $u=\phi\psi(w'+a_{t}\xi)$. Put $u':=\psi(w'+a_{t}\xi)$. We get the following commutative diagram
\[\xymatrix{& Q'\ar[r]^{f}\ar[d]_{u'} & Q & \\
X_{\alpha}\ar[r]_{a_{t+1}} & M_{t+1}\ar[r]_{b_{t+1}}\ar[ur]_{v\phi} & P_{t+1}\ar[r]_{\gamma\pi_{t+1}} & X_{\alpha}[1]
}
\]
Notice that $b_{t+1}u'=b_{t+1}\psi w'+b_{t+1}\psi a_{t}\xi=b_{t+1}\psi w'=g_{t}b_{t}w'$ belongs to $\ci$ since so does $g_{t}$. Apply the octahedron axiom
\[\xymatrix{Q'\ar[r]^{u'}\ar@{=}[d] & M_{t+1}\ar[d]^{b_{t+1}}\ar[r] & L\ar[r]\ar[d] & Q'[1]\ar@{=}[d] \\
Q'\ar[r] & P_{t+1}\ar[d]\ar[r] & N[1]\ar[d]^{n[1]}\ar[r]^{c[1]} & Q'[1]\ar[d]^{u'[1]} \\
& X_{\alpha}[1]\ar@{=}[r]\ar[d]_{-a_{t+1}[1]} & X_{\alpha}[1]\ar[d]\ar[r]_{-a_{t+1}[1]} & M_{t+1}[1] \\
& M_{t+1}[1]\ar[r] & L[1] &
}
\]
and consider the diagram
\[\xymatrix{N\ar[r]^{-c}&Q'\ar[r]^{b_{t+1}u'}\ar[d]^{f} & P_{t+1}\ar[r] & N[1] \\
&Q&&
}
\]
The morphism $-fc=-v\phi u'c=v\phi a_{t+1}n$ belongs to $\ci'$ (since so does $f$) and factors through $X_{\alpha}$. The hypothesis of induction implies that $-fc$ belongs to $\ci$. Since $b_{t+1}u'$ also belongs to $\ci$ and $\ci$ is saturated, then $f$ belongs to $\ci$.

\emph{Third step: Assume $\alpha$ is a limit ordinal and that every morphism of $\ci'$ factoring through $X_{\beta}$ with $\beta<\alpha$ belongs to $\ci$.} Then we have $X_{\alpha}=\lid_{\beta<\alpha}X_{\beta}$ and Lemma \ref{Stretching a filtration} ensures that we have a factorization
\[\xymatrix{&& Q'\ar[d]^{f}\ar@{.>}[dl]\ar@{.>}[dll] \\
X_{\beta}\ar[r]_{\can\ \ \ \ } & \lid_{\beta<\alpha}X_{\beta}\ar[r] & Q
}
\]
The hypothesis of induction implies that $f$ belongs to $\ci$.
\end{proof}

\end{document}